\documentclass[11pt]{article}

\usepackage[margin=1in]{geometry}
\usepackage{parskip}                
\usepackage{amsmath}                
\usepackage{amssymb}                
\usepackage{amsthm}
\usepackage{mathtools}              
\usepackage{bm}                     
\usepackage{array}                  
\usepackage{booktabs}
\usepackage{multirow}
\usepackage{xcolor}                 
\usepackage[square]{natbib}         
\usepackage[colorlinks=true,citecolor=blue!75!black,urlcolor=blue!75!black]{hyperref}  
\usepackage[capitalise]{cleveref}
\usepackage{algorithm}
\usepackage{algpseudocode}
\usepackage{tikz}
\usepackage[shortlabels]{enumitem}

\usepackage{siunitx}
\allowdisplaybreaks

\newcommand{\set}[1]{\left\{#1\right\}}

\def\N{{\mathbb{N}}}

\def\P{{\mathbb{P}}}

\def\R{{\mathbb{R}}}
\def\S{{\mathbb{S}}}

\def\bA{{\bm{A}}}

\def\bC{{\bm{C}}}

\def\bX{{\bm{X}}}

\def\cB{{\cal B}}

\def\cN{{\cal N}}

\def\cS{{\cal S}}
\def\cT{{\cal T}}

\def\cX{{\cal X}}
\def\cY{{\cal Y}}
\def\cZ{{\cal Z}}

\DeclareMathOperator{\Opt}{Opt}

\DeclareMathOperator{\rank}{rank}
\DeclareMathOperator{\Diag}{Diag}

\DeclareMathOperator{\Proj}{Proj}

\newcommand{\ones}{{\bm{1}}}

\newcommand{\x}{{\bm x}}
\newcommand{\z}{{\bm z}}
\newcommand{\y}{{\bm y}}

\newcommand{\g}{{\bm g}}

\newcommand{\w}{{\bm w}}

\newcommand{\q}{{\bm q}}

\renewcommand{\a}{{\bm a}}

\renewcommand{\b}{{\bm b}}

\newcommand{\stt}{\textrm{ s.t.\ }}

\newtheorem{assumption}{Assumption}
\crefname{assumption}{Assumption}{Assumptions}

\newtheorem{lemma}{Lemma}
\newtheorem{proposition}{Proposition}

\theoremstyle{definition}
\newtheorem{definition}{Definition}
\theoremstyle{remark}
\newtheorem{remark}{Remark}

\usepackage{authblk}

\begin{document}

\title{Parameter-Free Non-Ergodic Extragradient Algorithms for Solving Monotone Variational Inequalities\thanks{This research was supported in part by AFOSR [Grant FA9550-22-1-0365].}}

\author[1]{Lingqing Shen}
\author[1]{Fatma K{\i}l{\i}n\c{c}-Karzan}
\affil[1]{Tepper School of Business, Carnegie Mellon University}

\date{\today}

\maketitle

\begin{abstract}
Monotone variational inequalities (VIs) provide a unifying framework for convex minimization, equilibrium computation, and convex-concave saddle-point problems.  Extragradient-type methods are among the most effective first-order algorithms for such problems, but their performance hinges critically on stepsize selection. While most existing theory focuses on ergodic averages of the iterates, practical performance is often driven by the significantly stronger behavior of the last iterate. Moreover, available last-iterate guarantees typically rely on fixed stepsizes chosen using problem-specific global smoothness information, which is often difficult to estimate accurately and may not even be applicable. In this paper, we develop parameter-free extragradient methods with non-asymptotic last-iterate guarantees for constrained monotone VIs. 
For globally Lipschitz operators, our algorithm achieves an $o(1/\sqrt{T})$ last-iterate rate. We then extend the framework to locally Lipschitz operators via backtracking line search and obtain the same rate while preserving parameter-freeness, 
thereby making parameter-free last-iterate methods applicable to important problem classes for which global smoothness is unrealistic. 
Our numerical experiments on bilinear matrix games, LASSO, minimax group fairness, and state-of-the-art maximum entropy sampling relaxations demonstrate wide applicability of our results as well as strong last-iterate performance and significant improvements over existing methods.
\end{abstract}

\section{Introduction}

Variational inequalities {(VIs) with monotone operators} provide a unifying framework for a broad range of problems in optimization, including convex minimization, fixed-point problems, Nash equilibria, and convex-concave saddle-point problems. Among these applications, saddle-point problems have become particularly prominent in modern machine learning, arising in settings such as generative adversarial networks (GANs), robust reinforcement learning, and adversarial learning models.

A central challenge in first-order methods for monotone variational inequalities is stepsize selection. Classical methods typically rely on problem-dependent quantities such as Lipschitz constant of the operator or domain diameter in order to choose a valid stepsize. In practice, however, such quantities are often unavailable or difficult to estimate sharply, and relying on their conservative estimates can lead to excessively small stepsizes and slow convergence. Moreover, the local curvature of the operator may vary substantially across the domain, and so global Lipschitz constants often fail to reflect the behavior of the problem near a solution. As a result, there is now growing interest on \emph{parameter-free} or adaptive methods. 

At the same time, most available convergence rate guarantees are on  \emph{ergodic averages} of the iterates, whereas the practical performance of the \emph{non-ergodic (last-iterate)} algorithms is remarkably better (see \citep{zhu_new_2022,luo_adaptive_2025}). Also, last-iterate convergence is particularly important in applications where averaging may destroy desirable structural properties, such as sparsity or low rank. While non-asymptotic last-iterate guarantees have recently become available for certain classical methods under known smoothness assumptions, analogous guarantees for parameter-free methods remain largely unavailable.

In this paper, we address this gap by developing parameter-free extragradient-type algorithms that do not require prior knowledge of the Lipschitz constant, establishing their non-asymptotic last-iterate convergence rate of $o(1/\sqrt{T})$ for monotone VIs with either globally or  locally Lipschitz operators, and illustrating their practical efficacy on a variety of problem classes.

\subsection{Related work}

First-order methods for monotone VIs and convex--concave saddle-point problems dates back to \texttt{Gradient Descent-Ascent (GDA)}, which does not guarantee last-iterate convergence even in simple bilinear games because the iterates may cycle or spiral away from the saddle point \citep{arrow_constraint_1961}.  In a major advance, \cite{nemirovski_problem_1983} showed that convergence can be achieved by averaging the iterates of \texttt{Saddle-Point Mirror Descent (SP-MD)} (a generalization of \texttt{GDA} designed to accommodate non-Euclidean geometries through Bregman divergences), yielding the optimal $O(1/\sqrt{T})$ ergodic rate  for SP problems with general Lipschitz SP functions. Another foundational development is the \texttt{Extragradient (EG)} algorithm proposed by \cite{korpelevich_extragradient_1976}, who proved asymptotic convergence of the actual iterates of the algorithm. Unlike standard descent methods, \texttt{EG} performs a \emph{look-ahead step} that better anticipates the operator's behavior, mitigating the cycling phenomena observed in \texttt{GDA} and enabling last-iterate convergence in practice. \cite{nemirovski_prox-method_2004} later generalized  \texttt{EG} to the \texttt{Mirror-Prox (MP)} algorithm in the Bregman setting and established an $O(1/T)$ ergodic convergence rate for  monotone VIs with Lipschitz continuous operators.

The standard  \texttt{EG}  algorithm assumes a Lipschitz continuous operator and employs a fixed stepsize determined by the reciprocal of the Lipschitz constant. As a result, the practical implementation of  \texttt{EG}  algorithm faces the previously mentioned parameter dependent stepsize selection challenges. 
Several adaptive variants of  \texttt{EG}  have been proposed to address these challenges. Early approaches by \cite{khobotov_modification_1987} and \cite{marcotte_application_1991} introduced adaptive rules based on backtracking line search that enforce Armijo-type conditions. Similarly, \cite{iusem_iterative_1994} proposed an adaptive scheme based on a bracketing procedure. By interpretting  \texttt{EG}  algorithm as a prediction-correction framework, \cite{he_improvements_2002} introduced separate stepsizes for the prediction and correction steps. Nevertheless, all these adaptive schemes for  \texttt{EG} algorithm rely on Fej\'er-type monotonicity arguments similar to those used in the original analysis, and thus establish \emph{only asymptotic convergence} of the generated iterates. Moreover, they  typically require iterative subprocedures at each iteration to determine an admissible stepsize, increasing the per-iteration computational cost and the overall complexity of the algorithm. 

Although the ergodic convergence theory of \texttt{EG} algorithm is now well understood, the results on non-asymptotic guarantees for the \emph{actual iterates} are rather scarce and focus on VIs with globally Lipschitz operators. A sequence of recent works has established last-iterate guarantees for the standard \texttt{EG}  method under constant stepsizes chosen from known smoothness parameters.
In the unconstrained setting,  \cite{golowich_last_2020} presented the first such result,  proving 
an  $O(1/\sqrt{T})$ last-iterate rate under the additional assumption that the Jacobian of the operator is Lipschitz continuous, and also established matching lower bounds.  This additional assumption was later removed by \cite{gorbunov_exgradient_2022}, and \cite{cai_tight_2022} extended the same last-iterate rate to constrained monotone VIs. More recently, \cite{antonakopoulos_extra-gradient_2024} proved a strictly faster rate of $o(1/\sqrt{T})$ last-iterate rate for \texttt{EG} algorithm, 
and \cite{upadhyaya_lyapunov_2026} obtained comparable rate guarantees through a Lyapunov-based analysis. 
With the exception of \cite{antonakopoulos_extra-gradient_2024},  however, all of these algorithms rely on stepsize selection methods that depend on prior knowledge of problem-specific global Lipschitz constants.
\cref{tab:last-iterate-literature} summarizes existing constant-stepsize \texttt{EG} variants with  last-iterate convergence rates together with the corresponding convergence metrics used; see \cref{sec:metric} for the definitions of these metrics.  

\begin{table}[htbp]
\centering
\begin{tabular}{lclr}
\toprule
\textbf{Reference} & \textbf{Domain} & \textbf{Convergence Metric} & \textbf{Rate} \\
\midrule
\cite{golowich_last_2020} & unconstrained & \texttt{gap}, \texttt{nat res} & $O(1/\sqrt{T})$ \\
\cite{gorbunov_exgradient_2022} & unconstrained & \texttt{gap}, \texttt{nat res} & $O(1/\sqrt{T})$ \\
\cite{cai_tight_2022} & constrained &\texttt{gap}, \texttt{nat res}, \texttt{tan res} & $O(1/\sqrt{T})$ \\
\cite{antonakopoulos_extra-gradient_2024} & constrained & \texttt{gap} & $o(1/\sqrt{T})$ \\
\cite{upadhyaya_lyapunov_2026} & constrained & \texttt{gap}, \texttt{nat res} & $o(1/\sqrt{T})$ \\
\bottomrule
\end{tabular}
\caption{Summary of last-iterate convergence rates of \texttt{EG} variants with constant stepsizes (\texttt{gap} = restricted gap function; \texttt{nat res} = natural residual; \texttt{tan res} = tangent residual) }
\label{tab:last-iterate-literature}
\end{table}

A separate line of work has focused on \emph{parameter-free or adaptive methods} that achieve non-asymptotic convergence guarantees without requiring explicit knowledge of Lipschitz constant for stepsize selection. Early contributions include
\texttt{Universal MP} \citep{bach_universal_2019}, which adapts automatically to both smooth and nonsmooth structures but requires bounded domains and operators, and
\texttt{Adaptive MP} \citep{antonakopoulos_adaptive_2019}, which uses a stepsize rule based on operator variation to handle singularities near the domain boundary. Building on these, 
\texttt{AdaProx} \citep{antonakopoulos_adaptive_2021} introduces a universal method for singular operators, providing ergodic rate interpolation for both Lipschitz and bounded operators, while also achieving asymptotic convergence for the actual iterates. 
Relatedly, the past-extragradient  framework improves efficiency by reusing previous operator evaluations; its adaptive counterpart, \texttt{AdaPEG} \citep{ene_adaptive_2022}, achieves ergodic rates for bounded operators across smooth and nonsmooth settings. Another approach that departs  from the \texttt{EG} framework is
\texttt{aGRAAL} \citep{malitsky_golden_2020}, which relies on only a single operator evaluation per iteration and uses a stepsize rule based on only the local Lipschitz constant estimates. 
However, the non-asymptotic convergence rates established for these methods are all ergodic. 
Moving toward non-ergodic guarantees, the \texttt{EG+ algorithm} with adaptive stepsizes \citep{bohm_solving_2023} addresses a broader class of VIs in the unconstrained setting, and provides a rate for the best operator norm of the iterates.
More recently, the \texttt{Adapt EG} \citep{antonakopoulos_extra-gradient_2024} has made significant progress by achieving a last-iterate rate of $o(1/\sqrt{T})$ in terms of the restricted gap function. As summarized in \cref{tab:adaptive-literature}, 
while parameter-free methods now enjoy strong ergodic guarantees under a variety of assumptions, non-asymptotic last-iterate guarantees are still very limited, with the only exception being \texttt{Adapt EG} which handles only globally Lipschitz operators and uses monotonically decreasing stepsizes.

Thus, a theoretical gap remains  between adaptive (parameter-free) non-monotone stepsize design and non-asymptotic last-iterate convergence theory. This gap is especially pronounced in the case of VIs with locally Lipschitz continuous operators: 
while existing parameter-free methods for this setting provide ergodic guarantees, to the best of our knowledge, there is currently no parameter-free extragradient framework with non-asymptotic last-iterate guarantees under such weak smoothness assumptions, thereby significantly limiting the theoretical and practical applicability of last-iterate theory in settings where global smoothness is unrealistic.

\begin{table}[htbp]
\centering
\resizebox{\textwidth}{!}{%
\begin{tabular}{lcccccc} 
\toprule
\textbf{Algorithm} & \textbf{Domain} & \textbf{Operator} & \textbf{Rate} & \textbf{Metric} & \textbf{Ergodicity} \\
\midrule
Universal MP [\citenum{bach_universal_2019}] & bounded & Lipschitz & \phantom{\textsuperscript{*}}$O(1/T)$\textsuperscript{*} & \texttt{gap} & ergodic \\ 
Adaptive MP [\citenum{antonakopoulos_adaptive_2019}] & constrained & Lipschitz & $O(1/T)$ & \texttt{gap} & ergodic \\ 
AdaProx [\citenum{antonakopoulos_adaptive_2021}] & constrained & Lipschitz & $O(1/T)$ & \texttt{gap} & ergodic \\ 
AdaPEG [\citenum{ene_adaptive_2022}] & constrained & Lipschitz & $O(1/T)$ & \texttt{gap} & ergodic \\ 
aGRAAL [\citenum{malitsky_golden_2020}] & constrained & local Lipschitz & \phantom{\textsuperscript{*}}$O(1/T)$\textsuperscript{*} & \texttt{gap} & ergodic \\ 
Adaptive EG+ [\citenum{bohm_solving_2023}] & unconstrained & Lipschitz & $O(1/\sqrt{T})$ & \texttt{nat res} & best-iterate \\ 
Adapt EG [\citenum{antonakopoulos_extra-gradient_2024}] & constrained & Lipschitz & $o(1/\sqrt{T})$ & \texttt{gap} & last-iterate \\ 
\cref{alg:extragradient-adaptive} & constrained & Lipschitz & $o(1/\sqrt{T})$ & \texttt{tan res} & last-iterate \\ 
\cref{alg:extragradient-backtracking} & constrained & local Lipschitz & $o(1/\sqrt{T})$ & \texttt{tan res} & last-iterate \\ 
\cref{alg:extragradient-backtracking-old} & constrained & local Lipschitz & $o(1/\sqrt{T})$ & \texttt{tan res} & last-iterate \\
\bottomrule
\end{tabular}
}
\caption{Summary of parameter-free algorithms with non-asymptotic convergence rates \\({\texttt{gap} = restricted gap function; \texttt{nat res} = natural residual}; \texttt{tan res} = tangent residual; * indicates rates that the corresponding algorithm additionally assumes a bounded operator)}
\label{tab:adaptive-literature}
\end{table}

\subsection{Contribution and outline}

In this paper, we address this gap by developing parameter-free extragradient-type algorithms with non-asymptotic last-iterate guarantees for monotone VIs with both globally or locally Lipschitz operators. Unlike existing methods that rely on monotonically decreasing stepsizes and global Lipschitz assumptions, our algorithms utilizes non-monotone stepsize schemes that are better at adapting to local geometry and recovering from poor initializations. 
As a result, our algorithms significantly advance both the theory and practice of solving  monotone VIs.

Our main contributions are as follows:

\begin{enumerate}[(i)]
    \item   We introduce the \emph{extragradient residual}, an optimality metric naturally aligned with extragradient updates, and show that it upper bounds the  tangent residual from the literature. This residual provides a tractable basis for establishing non-asymptotic last-iterate guarantees for adaptive extragradient-type methods.

    \item  For monotone VIs with globally Lipschitz operators, we propose \texttt{PF-NE-EG} (\cref{alg:extragradient-adaptive}), an extragradient method with an adaptive stepsize rule that eliminates the need for problem-specific parameters such as Lipschitz constants and domain diameters while employing a simple stepsize update rule with minimal computational overhead. We establish a non-asymptotic last-iterate convergence rate of $o(1/\sqrt{T})$ in extragradient residual. To the best of our knowledge, this is the first such guarantee for a parameter-free method in a metric bounding the tangent residual and demonstrating strong practical performance.

    \item  We extend our framework to locally Lipschitz monotone operators through \cref{alg:extragradient-backtracking,alg:extragradient-backtracking-old}, which adopt  backtracking line searches for stepsize selection. We prove that these backtracking line searches are well-defined and incur only finitely many failed reductions overall. These methods remain parameter-free and achieve the same non-asymptotic last-iterate rate of $o(1/\sqrt{T})$ under the much weaker locally Lipschitz operator assumption. To the best of our knowledge, this is the first last-iterate convergence guarantee for a parameter-free method in this nonsmooth setting.
    {Thus, we significantly extend the applicability of last-iterate guarantees to new problem classes;} recall that in the nonsmooth setting the only other parameter-free {algorithm is \texttt{aGRAAL}, which admits only an ergodic rate}; see \citep{malitsky_proximal_2018}.

    \item  We evaluate the proposed methods on four representative classes of problems: bilinear matrix games, SP formulation of LASSO, minimax group fairness, and SP formulation of a state-of-the-art maximum entropy sampling problem (MESP) relaxation. 
    {While the monotone VIs associated with bilinear matrix games and LASSO have globally Lipschitz continuous operators, minimax group fairness and MESP come with only locally Lipschitz operators.}
    Across all instances, the proposed algorithms exhibit strong last-iterate performance and compare quite favorably with existing baselines; {the improvements becoming exceptionally pronounced in the cases of VIs with locally Lipschitz operators.} 
    Notably, in the MESP setting, our approach provides the first convex optimization algorithm to solve the g-scaled linx and double-scaled linx relaxations as well as the first approach with theoretical convergence rate guarantees for mixing-based MESP bounds (see \cite{chen_mixing_2021,chen_computing_2023,chen_generalized_2024,ponte_admm_2025,shen_majorization_2026}). 
\end{enumerate}

The remainder of the paper is organized as follows. In \cref{sec:preliminaries}, we introduce notation, convergence metrics, and preliminary results. In \cref{sec:alg}, we present our algorithm, \texttt{PF-NE-EG}. For monotone VIs with globally Lipschitz operators,  we establish the last-iterate convergence guarantees of \texttt{PF-NE-EG} in \cref{sec:alg-smooth}. In \cref{sec:alg-backtracking-new}, we develop the backtracking variant with non-monotone stepsizes for locally Lipschitz operators and prove its convergence properties. We report our numerical study in \cref{sec:numerical}. We present another variant of \texttt{PF-NE-EG} with backtracking based monotone stepsizes in \cref{sec:appendix}. 
\subsection{Notation}

Given a positive integer $d$, let $[d]\coloneqq\set{1,2,\dots,d}$. Let $\ones\coloneqq(1,\dots,1)\in\R^d$ be the vector of all ones. 
Let $\Delta_d\coloneqq\set{\x\in\R^d_+:\ones^\top\x=1}$ be the standard simplex in $\R^d$. For $\x\in\R^d$, let $x_i$ be the $i$\textsuperscript{th} component of $\x$. Let $\|\x\|_2,~ \|\x\|_1$ and $\|\x\|_\infty$ denote the Euclidean, $\ell_1$ and $\ell_\infty$ norm of $\x$, respectively. {With slight abuse of notation, we denote $\exp(\x)\coloneq (\exp(x_1), \dots, \exp(x_d))$ to be the vector obtained by applying the exponential component-wise.}
Let $\S^d$ be the vector space of $d\times d$ real symmetric matrices and $\S^d_+$ the cone of positive semidefinite matrices. 
{For $\x\in\R^d$, we represent the diagonal matrix with diagonal entries $\x$ by $\Diag(\x)\in\S^d$. For $\bX\in\S^d$, we denote the logarithm of its determinant by $\log\det(\bX)$.}
Given a convex function $f:\R^d\to(-\infty,+\infty]$ and a vector $\x\in\R^d$, let $\partial f(\x)$ be the set of subdifferential of $f$ at the point $\x$, and $\nabla f(\x)$ a subgradient of $f$ at $\x$. Given a bivariate function $(\x,\y)\mapsto\phi(\x,\y)$ that is convex in $\x$ and concave in $\y$, let $\nabla_\x\phi(\x,\y)$ be the subgradient of $\phi$ w.r.t.\ $\x$ with $\y$ fixed, and $\nabla_\x\phi(\x,\y)$ be the supergradient of $\phi$ w.r.t.\ $\y$ with $\x$ fixed. {Let $\Proj_{\mathcal{Z}}(\mathbf{x})$ denote the orthogonal projection of $\x \in \R^d$ onto $\cZ\subset\R^d$. By convention, we define $\frac00=0$ whenever $0$ appears in the denominator.}

\section{Preliminaries}
\label{sec:preliminaries}

In this section, we formally define the problem as well as notation that will be used throughout the paper. We focus on monotone variational inequalities, with a particular emphasis on the application to convex-concave saddle-point problems.

\subsection{Variational inequalities}

Let $\cZ\subset\R^d$ be a nonempty, closed, and convex set, and let $F: \cZ \to \R^d$ be a continuous operator. A variational inequality (VI) problem associated with $\cZ$ and $F$ is to find a vector $\z_* \in \cZ$ such that
\begin{align}\label{eq:vi}
  \langle F(\z_*), \z - \z_* \rangle \geq 0, \quad \forall \z \in \cZ. \tag{VI}
\end{align}
{Such a solution $\z_*$ is called a \emph{strong solution} of the VI.} 
We are interested in the case where $F$ is \emph{monotone}, i.e.,
\begin{align}\label{eq:monotone}
  \langle F(\z) - F(\w), \z - \w \rangle \geq 0, \quad \forall \z, \w \in \cZ.
\end{align}
{The monotonicity of the operator corresponds to the convexity of the underlying problem, and it} is a standard assumption in the study of VIs. 

An important instance of \eqref{eq:vi} that motivates our work is the \emph{convex-concave saddle point problem}
\begin{align}\label{eq:spp}
  \min_{\x \in \cX} \max_{\y \in \cY} \phi(\x, \y),  \tag{SP}
\end{align}
where $\cX,\ \cY$ are closed convex sets, and $\phi: \cX \times \cY \to \R$ is convex in $\x$ for every $\y$ and concave in $\y$ for every $\x$. A point $(\x^*, \y^*) \in \cX\times\cY$ is a \emph{saddle point} of $\phi$ if it satisfies  
\[ \phi(\x^*, \y) \leq \phi(\x^*, \y^*) \leq \phi(\x, \y^*), \quad \forall \x \in \cX,\; \forall \y \in \cY. \] 
Problem \eqref{eq:spp} can be equivalently formulated as a variational inequality of the form  \eqref{eq:vi} by defining $\z \coloneq (\x, \y) \in \cZ \coloneq \cX \times \cY$ and the operator  $F(\z) \coloneq (\nabla_\x \phi(\x, \y), -\nabla_\y \phi(\x, \y))$. Then, by the convex-concave structure of $\phi$, we immediately deduce that $F(\z)$ is a monotone operator satisfying \eqref{eq:monotone}. Moreover, a point $\z_*=(\x_*,\y_*)$ is a saddle point of $\phi$ if and only if it solves \eqref{eq:vi}, i.e., the saddle points of $\phi$ correspond exactly to the solutions of the associated variational inequality.

\subsection{Convergence metrics}\label{sec:metric}

The most common convergence metric for \eqref{eq:vi} is the \emph{restricted gap function}, which is particularly useful for establishing the convergence of ergodic iterates.
\begin{definition}[Restricted gap function]
  For {a given solution $\z\in\cZ$ and} any compact set $\cB\subset\cZ$, the restricted gap function is defined as
  \begin{align*}
    G_{\cB}(\z) \coloneqq \sup_{\w \in {\cB}} \langle F(\w), \z - \w \rangle. 
  \end{align*}
\end{definition}
In the context of \eqref{eq:spp}, this corresponds to the classical convergence metric, \emph{the restricted saddle point gap}, which provides a certificate of optimality. Given \eqref{eq:spp}, there are a pair of associated primal and dual problems: 
\begin{align*}
  \Opt(P) \coloneqq \min_{\x\in\cX}\overline\phi(\x) &= \min_{\x\in\cX}\max_{\y\in\cY}\phi(\x,\y), \\
  \Opt(D) \coloneqq \max_{\y\in\cY}\underline{\phi}(\y) &= \max_{\y\in\cY}\min_{\x\in\cX}\phi(\x,\y). 
\end{align*}
Under strong duality, $\Opt(P)=\Opt(D)$. If the restriction set $\cB=\cX_B\times\cY_B\subset\cX\times\cY$ is sufficiently large to contain the optimal saddle point, then the restricted saddle point gap can be decomposed into the sum of the restricted primal and dual optimality gaps: 
\begin{align*}
  G_{\cB}(\x, \y)
  &= \max_{\y' \in \cY_B} \phi(\x, \y') - \min_{\x' \in \cX_B} \phi(\x', \y) \\
  &= \max_{\y' \in \cY_B} \phi(\x, \y') - \Opt(P) + \Opt(D) - \min_{\x' \in \cX_B} \phi(\x', \y). 
\end{align*}

The restriction on a compact set is essential for problems defined on unbounded domains, because the gap can go to infinity even in the simple case of bilinear matrix games. 
While the gap function is useful for studying theoretical convergence guarantees, it can be difficult (or not computationally efficient) to evaluate in practice due to its reliance on solving optimization problems. 

The \emph{natural residual} is another standard convergence metric for \eqref{eq:vi} that addresses the computational limitations of the gap function. 
\begin{definition}[Natural residual]
  For {a given solution $\z\in\cZ$ and} a given stepsize $\eta > 0$, the natural residual is defined as
  \begin{align*}
    R_{\eta}(\z) \coloneq \tfrac{1}{\eta} \| \z - \Proj_{\cZ}(\z-\eta F(\z)) \|_2. 
  \end{align*}
\end{definition}
As a convergence metric, the natural residual satisfies $R_{\eta}(\z) = 0$ if and only if $\z$ is a solution to \eqref{eq:vi}. 
Unlike the gap function which averages out oscillations for monotone operators, the natural residual is better suited for measuring the actual iterates. Moreover, it takes little computational effort to evaluate $R_{\eta_t}(\z_t)$ since its computation does not require solving optimization problems. 
Furthermore, it remains well-defined for unbounded domains and better captures the local properties, and it provides an upper bound of the restricted gap (see e.g., \cite[Lemma 2]{cai_tight_2022}) up to a factor of the restriction set diameter.

The \emph{tangent residual}, recently introduced by \cite{cai_tight_2022}, proves to be particularly useful for analyzing the last-iterate convergence of extragradient-type algorithms (e.g., see \cite{tran-dinh_sublinear_2023}). 
\begin{definition}[Tangent residual]\label{def:tangent-residual}
  {For a given solution $\z\in\cZ$,} the tangent residual is defined as
  \begin{align*}
    T(\z) \coloneq \|F(\z) + \Proj_{\cN_\cZ(\z)}(-F(\z))\|_2,   
  \end{align*}
  where $\cN_{\cZ}(\z)\coloneq\set{\bm\xi\in\R^d:\langle\bm\xi,\w-\z\rangle\leq0, \forall\w\in\cZ}$ is the normal cone of $\cZ$ at $\z$. 
\end{definition}
Note that {as $\cZ$ is a nonempty closed convex set, we have that $\cN_{\cZ}(\z)$ is a closed convex cone for every $\z\in\cZ$. Then, by its definition, the tangent residual satisfies $T(\z)=\min_{\bm\xi}\{\|F(\z)+\bm\xi\|_2:\, \bm\xi\in \cN_{\cZ}(\z)\}$. 
Thus, if} $\z$ is in the interior of $\cZ$, then $\cN_{\cZ}(\z)=\set{0}$ and $T(\z)=\|F(\z)\|_2$ is reduced to the operator norm, which is a common convergence metric for unconstrained problems. Let $\cT_\cZ(\z)\coloneq\set{\bm\xi\in\cS:\langle\bm\xi,\w\rangle\leq0, \forall\w\in\cZ}$ be the tangent cone of $\cZ$ at $\z$. Recall that $\cT_{\cZ}(\z)$  is the polar of $\cN_\cZ(\z)$, and by Moreau decomposition \citep{moreau_decomposition_1962,combettes_moreau_2013}, $\w=\Proj_{\cN_\cZ(\z)}(\w) + \Proj_{\cT_\cZ(\z)}(\w)$ for all $\w\in\cZ$. Thus, $T(\z)=\|\Proj_{\cT_\cZ(\z)}(-F(\z))\|_2$. This means that, if $\z$ is on the boundary, then $T(\z)$ measures the projection of $-F(\z)$ onto the tangent cone $\cT_\cZ(\z)$, which can be viewed intuitively as the steepest descent in a feasible direction. 

As a convergence metric, the tangent residual has the basic property that $\z_*$ is an optimal solution to \eqref{eq:vi} if and only if $T(\z_*)=0$. This is because \eqref{eq:vi} can be equivalently written as $-F(\z_*)\in\cN_\cZ(\z_*)$ by definition of $\cN_\cZ(\cdot)$. Furthermore, the tangent residual is an upper bound for the natural residual (see \cite[Lemma 1]{cai_tight_2022}). It also has the advantage of not relying on any parameters such as compact set $\cB$ or stepsize $\eta$ in its definition. Although computationally it is not as tractable as the natural residual, it is especially useful for the analysis of extragradient-type algorithms, in which case it admits a nice monotonicity property. {When the underlying VI originates from constrained optimization, i.e.,} ($F=\nabla f$), {$T(\z)$} directly captures the KKT stationarity error of the Lagrangian.

Finally, we present two basic facts that will be used throughout our analysis. 

\begin{lemma}[{3-point identity}]\label{lem:3-point-identity}
  For any three points $\a,\b,\bm c\in\cZ$, 
  \begin{align*}
    \|\a-\bm c\|_2^2 + \|\a-\b\|_2^2 - \|\b-\bm c\|_2^2 = {2}\langle\b-\a, \bm c-\a\rangle. 
  \end{align*}
\end{lemma}

\begin{lemma}[{Olivier's Theorem \citep[p. 124]{knopp_theory_1990}}]\label{lem:olivier}
  Let $\{a_n\}_{n\in\N}$ be a non-increasing sequence of positive numbers such that $\sum_{n=1}^{\infty} a_n < +\infty $ for all $n\in\N$.
  Then $a_n = o\left(1/n\right)$.
\end{lemma}

\section{Parameter-free non-ergodic variational inequality algorithm}
\label{sec:alg}

In this section, we develop and analyze two parameter-free extragradient algorithms for solving \eqref{eq:vi}. 
Our proposed algorithms are built upon the standard extragradient algorithm \citep{korpelevich_extragradient_1976} given by the following update rule:
\begin{align}
  \w_t &= \Proj_{\cZ}(\z_t-\eta_tF(\z_t)), \label{eq:extragradient-1} \\
  \z_{t+1} &= \Proj_{\cZ}(\z_t-\eta_tF(\w_t)). \label{eq:extragradient-2} 
\end{align}
Our algorithms deviate from the original form by {our design of novel parameter-free} stepsize schemes that {result in both theoretical convergence guarantees on the last iterate and also achieve  practical efficiency.} Instead of using a constant stepsize $\eta_t=\eta$ that relies on the reciprocal of the global Lipschitz constant, our stepsize is based on the estimate of local Lipschitz constant around each iterate. Unlike AdaGrad-style stepsizes, our scheme can be easily extended to backtracking line search, and {hence can} naturally exploit the local {Lipschitz continuity} structure of $F$ without requiring global Lipschitz continuity. In addition, {our analysis can accommodate} for non-monotonically decreasing stepsizes and thus better adjusts to the local curvature.

To facilitate the discussion, we define the local Lipschitz estimates at iteration $t$ as follows:
\begin{align}
  L_t &\coloneqq \frac{\|F(\w_t)-F(\z_t)\|_2}{\|\w_t-\z_t\|_2}, \label{eq:local-Lipschitz-estimate-1}\\
  \hat L_t &\coloneqq \begin{cases}
    \frac{\|F(\w_t)-F(\z_{t+1})\|_2}{\|\w_t-\z_{t+1}\|_2}, & \text{if } \w_t\neq\z_{t+1} \\
    0, & \text{if }\w_t=\z_{t+1}. 
  \end{cases}\label{eq:local-Lipschitz-estimate-2}
\end{align}

\begin{remark}\label{rem:L-def}
  If $\w_t=\z_t$, by \eqref{eq:extragradient-1} and the property of projection, $\langle-\eta_t F(\z_t),\z-\z_t\rangle\leq0$ for any $\z\in\cZ$. By monotonicity of $F$, this further shows that $\z_t$ is an optimal solution of \eqref{eq:vi}. Therefore, whenever $\w_t=\z_t$, we may stop the algorithm and conclude optimality. In the analysis to follow, we will focus on an iteration $t$ before the termination of the algorithm, ensuring $L_t$ is well-defined. 
\end{remark}

\begin{remark}
  The definition of $\hat L_t$ ensures $\|F(\w_t)-F(\z_{t+1})\|_2 = \hat L_t \|\w_t-\z_{t+1}\|_2$.
\end{remark}

\begin{remark}
  Under \cref{assum:smooth}, we have $L\geq\max\set{L_t,\hat L_t}$ for all $t$. 
\end{remark}

Before proceeding to the analysis, we introduce {our new convergence metric.
\begin{definition}[Extragradient residual]
  For {a given solution $\z\in\cZ$ and} a given stepsize $\eta > 0$, the extragradient residual is defined to be
  \[ \|F(\z_{t+1})+\bm\xi_{t+1}\|_2, \] 
  where
\begin{align}\label{eq:def:ex-residual}
  \bm\xi_{t+1} \coloneqq -\frac{1}{\eta_{t}}[\z_{t+1} - (\z_{t} - \eta_{t}F(\w_{t}))]. 
\end{align}
\end{definition}
}
The extragradient residual $\|F(\z_{t+1})+\bm\xi_{t+1}\|_2$ vanishes if and only if $\z_{t+1}$ is a solution of \eqref{eq:vi}. 
Intuitively, $F(\z_{t+1})$ captures the local behavior of the operator, while $\bm\xi_{t+1}$ represents the projection residual of the extragradient step \eqref{eq:extragradient-2} and captures the displacement needed to maintain feasibility within the constraint set $\cZ$.
Although this residual is limited to extragradient-type algorithms and last-iterate analysis, it is a stronger convergence metric than the tangent or natural residuals, defined in \cref{sec:metric}, commonly used in the literature (see \cref{lem:eg-res-geq-tangent-res}). Moreover, whenever extragradient residual is applicable, its computational overhead is rather small. {In our theoretical results, we will use the extragradient residual $\|F(\z_{t})+\bm\xi_{t}\|_2$ as our primary convergence metric.}

\begin{lemma}\label{lem:eg-res-geq-tangent-res}
  Let $\z_{t+1}$ be generated by the extragradient step \eqref{eq:extragradient-1}--\eqref{eq:extragradient-2}. Then the extragradient residual satisfies:
  \begin{align*}
  \|F(\z_{t+1}) + \bm\xi_{t+1}\|_2 \geq T(\z_{t+1}).
  \end{align*}
\end{lemma}

\begin{proof}
By the optimality conditions of the projection step in the extragradient update, we have $\bm\xi_{t+1} = -\frac{1}{\eta_{t}}[\z_{t+1} - (\z_{t} - \eta_{t}F(\w_{t}))] \in \cN_{\cZ}(\z_{t+1})$ (see \citep[Example 6.16]{rockafellar_variational_1998}). By \cref{def:tangent-residual}, the tangent residual at $\z_{t+1}$ is 
\begin{align*}
T(\z_{t+1}) 
&= \|F(\z_{t+1}) + \Proj_{\cN_\cZ(\z_{t+1})}(-F(\z_{t+1}))\|_2 \\
&\leq \|F(\z_{t+1}) + \bm\xi_{t+1}\|_2, 
\end{align*}
where the inequality follows from the definition of projection and $\bm\xi_{t+1} \in \cN_{\cZ}(\z_{t+1})$. 
\end{proof}

With these definitions in hand, we first derive some lemmas useful for the analysis of the extragradient algorithm updates \eqref{eq:extragradient-1}--\eqref{eq:extragradient-2}. This per-iteration analysis is independent of the stepsize selection scheme, as long as the stepsize $\eta_t$ satisfies certain conditions. As such, it will serve as the primary building block for the convergence proofs that follow.

\begin{lemma}\label{lem:descent}
  The update \eqref{eq:extragradient-1}--\eqref{eq:extragradient-2} of the extragradient algorithm with stepsize $\eta_t>0$ guarantees the following inequality {for any solution $\z_* \in \cZ$ of \eqref{eq:vi}}
  \begin{align*}
    \|\z_{t+1} - \z_*\|_2^2 \leq \|\z_t-\z_*\|_2^2 - (1-\eta_tL_t)\left[\|\z_t-\w_t\|_2^2 + \|\z_{t+1} - \w_t\|_2^2\right] .
  \end{align*} 
\end{lemma}

\begin{proof} 
  By the projection operations in \eqref{eq:extragradient-1}--\eqref{eq:extragradient-2}, we have
  \begin{align}
    \langle\w_t-(\z_t-\eta_tF(\z_t)), \w_t-\z\rangle &\leq 0, \label{eq:proj-ineq-1} \\
    \langle\z_{t+1}-(\z_t-\eta_tF(\w_t)), \z_{t+1}-\z\rangle &\leq 0, \label{eq:proj-ineq-2}
  \end{align}
  for all $\z\in\cZ$. Since $\z_*$ is a solution of \eqref{eq:vi} and $F$ is a monotone operator, we have 
  \begin{align}\label{eq:vi-opt-ineq}
    \langle F(\z), \z - \z_* \rangle \geq \langle F(\z_*),\z-\z_*\rangle\geq 0
  \end{align}
  for all $\z \in \cZ$. 
  By \cref{lem:3-point-identity}, 
  \begin{align*}
    \|\z_{t+1}-\z_*\|_2^2 = \|\z_t-\z_*\|_2^2 - \|\z_{t+1}-\z_t\|_2^2 - {2}\langle\z_t-\z_{t+1},\z_{t+1}-\z_*\rangle. 
  \end{align*}
  For the inner product, we have 
  \begin{align*}
    &\hspace{-2em}\langle\z_t-\z_{t+1}, \z_{t+1}-\z_*\rangle \\
    &\geq \eta_t\langle F(\w_t), \z_{t+1}-\z_*\rangle \\
    &= \eta_t\langle F(\w_t), \z_{t+1}-\w_t\rangle + \eta_t\langle F(\w_t), \w_{t}-\z_*\rangle \\
    &\geq \eta_t\langle F(\w_t), \z_{t+1}-\w_t\rangle \\
    &\geq \eta_t\langle F(\w_t), \z_{t+1}-\w_t\rangle - \langle\w_t-(\z_t-\eta_tF(\z_t)), \z_{t+1}-\w_t\rangle \\
    &= \eta_t\langle F(\w_t) - F(\z_t), \z_{t+1}-\w_t\rangle + \langle\z_t-\w_t, \z_{t+1}-\w_t\rangle \\
    &\geq -\eta_t\|F(\w_t) - F(\z_t)\|_2\|\z_{t+1}-\w_t\|_2 + \langle\z_t-\w_t,\z_{t+1}-\w_t\rangle \\
    &= -\eta_tL_t\|\z_t-\w_t\|_2\|\z_{t+1}-\w_t\|_2 + \langle\z_t-\w_t,\z_{t+1}-\w_t\rangle \\
    &\geq  -\tfrac12\eta_tL_t[\|\z_t-\w_t\|_2^2 + \|\z_{t+1}-\w_t\|_2^2] + \langle\z_t-\w_t,\z_{t+1}-\w_t\rangle. 
  \end{align*}
  The first inequality follows from \eqref{eq:proj-ineq-2} with $\z=\z_*$, the second inequality follows from taking $\z=\w_t\in\cZ$ in \eqref{eq:vi-opt-ineq} and $\eta_t\geq0$. The third inequality follows from \eqref{eq:proj-ineq-1} with $\z=\z_{t+1}\in\cZ$. The fourth inequality follows from Cauchy-Schwarz inequality. The last equality follows from the definition \eqref{eq:local-Lipschitz-estimate-1} of $L_t$. The last inequality follows from the identity $2ab\leq a^2+b^2$. 
  Applying \cref{lem:3-point-identity} again, the last inner product term can be rewritten as
  \begin{align*}
    2\langle\z_t-\w_t,\z_{t+1}-\w_t\rangle = \|\z_t-\w_t\|_2^2 + \|\z_{t+1}-\w_t\|_2^2 - \|\z_t - \z_{t+1}\|_2^2. 
  \end{align*}
  Putting things together, we obtain
  \begin{align*}
    \|\z_{t+1}-\z_*\|_2^2
    &= \|\z_t-\z_*\|_2^2 - \|\z_{t+1}-\z_t\|_2^2 - 2\langle\z_t-\z_{t+1}, \z_{t+1}-\z_*\rangle \\
    &\leq \|\z_t-\z_*\|_2^2 - \|\z_{t+1}-\z_t\|_2^2 + \eta_tL_t\left[\|\z_t-\w_t\|_2^2 + \|\z_{t+1}-\w_t\|_2^2\right] \\
    &\qquad - 2\langle\z_t-\w_t,\z_{t+1}-\w_t\rangle \\
    &= \|\z_t-\z_*\|_2^2 - (1 - \eta_tL_t)[\|\z_t-\w_t\|_2^2 + \|\z_{t+1}-\w_t\|_2^2]. 
  \end{align*}
\end{proof}

\cref{lem:descent} is a descent lemma that characterizes the change in the distance $\|\z_t-\z_*\|_2$ to the optimal solution from iteration to iteration. Next, we relate this to a measure of optimality. To this end, the following lemma provides an upper bound on our primary convergence metric the extragradient residual $\|F(\z_{t}) + \bm\xi_{t}\|_2$.

\begin{lemma}\label{lem:residual-bound}
  Suppose the stepsize $\eta_t{>0}$ satisfies $\eta_tL_t<1$, where $L_t$ is defined in \eqref{eq:local-Lipschitz-estimate-1}. Then, the update \eqref{eq:extragradient-1}--\eqref{eq:extragradient-2} of the extragradient algorithm with stepsize $\eta_t$ guarantees the following inequality:
  \begin{align*}
    \eta_t^2\|F(\z_{t+1}) + \bm\xi_{t+1}\|_2^2 \leq \frac{3+2\eta_t^2\hat L_t^2}{1-\eta_tL_t} \left[\|\z_t-\z_*\|_2^2 - \|\z_{t+1}-\z_*\|_2^2 \right]. 
  \end{align*}
\end{lemma}

\begin{proof}
  Recall by definition that $\bm\xi_{t+1}=-\frac{1}{\eta_{t}}[\z_{t+1} - (\z_{t} - \eta_{t}F(\w_{t}))]$. Thus, we have
  \begin{align*}
    & \eta_t^2\|F(\z_{t+1}) + \bm\xi_{t+1}\|_2^2 \\
    &= \|\eta_tF(\z_{t+1}) + \z_{t}-\z_{t+1}-\eta_tF(\w_t)\|_2^2 \\
    &\leq \left[\|\eta_t(F(\z_{t+1})-F(\w_t))\|_2 + \|\z_t-\w_t\|_2 + \|\z_{t+1}-\w_t\|_2 \right]^2 \\
    &\leq \tfrac{3+2\eta_t^2\hat L_t^2}{2\eta_t^2\hat L_t^2}\cdot\eta_t^2\|F(\z_{t+1})-F(\w_t)\|_2^2 + \tfrac{3+2\eta_t^2\hat L_t^2}{1}\|\z_t-\w_t\|_2^2 + \tfrac{3+2\eta_t^2\hat L_t^2}{2}\|\z_{t+1}-\w_t\|_2^2 \\
    &\leq \tfrac{3+2\eta_t^2\hat L_t^2}{2}\|\z_{t+1}-\w_t\|_2^2 + \tfrac{3+2\eta_t^2\hat L_t^2}{1}\|\z_t-\w_t\|_2^2 + \tfrac{3+2\eta_t^2\hat L_t^2}{2}\|\z_{t+1}-\w_t\|_2^2 \\
    &= (3+2\eta_t^2\hat L_t^2)[\|\z_t-\w_t\|_2^2 + \|\z_{t+1}-\w_t\|_2^2] \\
    &\leq \frac{3+2\eta_t^2\hat L_t^2}{1-\eta_tL_t}[\|\z_t-\z_*\|_2^2 - \|\z_{t+1}-\z_*\|_2^2]. 
  \end{align*}
  Here, the second inequality follows from Titu's lemma: Given $a_i\in\R$ and $b_i>0$ for $i\in[n]$, we have
  \[ \frac{(a_1+\cdots+a_n)^2}{b_1+\cdots+b_n} \leq \frac{a_1^2}{b_1} + \cdots + \frac{a_n^2}{b_n}. \]
  The third inequality holds due to the definition of $\hat L_t$. 
  The last inequality follows by applying \cref{lem:descent} and noting that $\eta_tL_t<1$. 
\end{proof}

The next lemma establishes the monotonicity of the extragradient residual $\|F(\z_{t}) + \bm\xi_{t}\|_2$ under the condition $\eta_t\hat L_t\leq1$. This is a crucial property for proving the convergence of the actual iterates rather than the ergodic average. 

\begin{lemma}\label{lem:residual-nonincreasing}
  Suppose the stepsize {$\eta_t>0$} satisfies $\eta_t\hat L_t\leq1$. 
  Then the update \eqref{eq:extragradient-1}--\eqref{eq:extragradient-2} of the extragradient algorithm guarantees that $\|F(\z_t) + \bm\xi_t\|_2$ is non-increasing {as $t$ increases}. 
\end{lemma}

\begin{proof}
  Let us denote the residual of the first projection operation in \eqref{eq:extragradient-1} by $\bm\zeta_t\coloneqq-\frac{1}{\eta_t}[\w_t-(\z_t-\eta_tF(\z_t))]$   so that {together with our extragradient residual $\bm\xi_{t+1}$ from \eqref{eq:def:ex-residual} we have}
  \begin{align*}
    \w_t &= \z_t-\eta_t(F(\z_t) + \bm\zeta_t), \\
    \z_{t+1} &= \z_t-\eta_t(F(\w_t) + \bm\xi_{t+1}). 
  \end{align*}
  In addition, define 
  \begin{align*}
    \g_t &\coloneqq F(\z_{t}) + \bm\xi_{t}, \\
    \tilde{\g}_t &\coloneqq F(\z_t) + \bm\zeta_t=-\tfrac{1}{\eta_t}(\w_t-\z_t), \\
    \hat{\g}_t &\coloneqq F(\w_t) + \bm\xi_{t+1} = -\tfrac{1}{\eta_t}(\z_{t+1}-\z_t).
  \end{align*}
  Thus, we would like to show $\|g_{t+1}\|_2 \leq \|g_t\|_2$. 

  Noting that {$\eta_t\bm\zeta_t$ and $\eta_t\bm\xi_t$} are residuals of the projection operation in the update steps \eqref{eq:extragradient-1} and \eqref{eq:extragradient-2} respectively, we have $\langle\bm\zeta_{t},\w_{t}-\z\rangle\geq0$, $\langle\bm\xi_{t},\z_{t}-\z\rangle\geq0$ for any $\z\in\cZ$. Therefore,  
  \begin{align*}
    0\leq \langle\bm\xi_{t+1},\z_{t+1}-\z_t\rangle &= \langle\g_{t+1} - F(\z_{t+1}),\z_{t+1}-\z_t\rangle \leq \langle\g_{t+1} - F(\z_t),\z_{t+1}-\z_t\rangle ,  
  \end{align*}
  where the second inequality follows from the monotonicity of $F$, i.e., $\langle F(\z_{t+1}) - F(\z_t), \z_{t+1}-\z_t\rangle\geq0$. 
  Moreover, using the definitions of $\g_t$ and $\tilde \g_t$ we arrive at
  \begin{align*}
    0\leq \langle\bm\xi_{t},\z_{t}-\w_t\rangle &= \langle\g_t-F(\z_t),\z_{t}-\w_t\rangle, \\
    0\leq \langle\bm\zeta_{t},\w_{t}-\z_{t+1}\rangle &= \langle\tilde\g_t - F(\z_t),\w_{t}-\z_{t+1}\rangle \\
    &= \langle\tilde\g_t - F(\z_t),\w_{t}-\z_{t}\rangle + \langle\tilde\g_t - F(\z_t),\z_{t}-\z_{t+1}\rangle. 
  \end{align*}
  Summing up the preceding three inequalities leads to
  \begin{align*}
    0 &\leq \langle\bm\g_{t+1}-\tilde\g_t,\z_{t+1}-\z_t\rangle + \langle\g_t-\tilde\g_t,\z_{t}-\w_t\rangle \\
    &= -\eta_t\langle\bm\g_{t+1}-\tilde\g_t,\hat\g_t\rangle + \eta_t\langle\g_t-\tilde\g_t,\tilde\g_t\rangle \\
    &= -\eta_t\langle\bm\g_{t+1},\hat\g_t\rangle + \eta_t\langle\tilde\g_t,\hat\g_t\rangle+ \eta_t\langle\g_t,\tilde\g_t\rangle - \eta_t\|\tilde\g_t\|_2^2 \\
    &= \tfrac12\eta_t[\|\g_{t+1}-\hat\g_t\|_2^2 - \|\g_{t+1}\|_2^2 - \|\hat\g_t\|_2^2] - \tfrac12\eta_t[\|\tilde\g_t-\hat\g_t\|_2^2 - \|\tilde\g_t\|_2^2 - \|\hat\g_t\|_2^2] \\
    &\qquad - \tfrac12\eta_t[\|\g_t-\tilde\g_t\|_2^2-\|\g_t\|_2^2 - \|\tilde\g_t\|_2^2] - \eta_t\|\tilde\g_t\|_2^2 \\
    &= \tfrac12\eta_t[\|\g_{t+1}-\hat\g_t\|_2^2 - \|\g_{t+1}\|_2^2 - \|\tilde\g_t-\hat\g_t\|_2^2 - \|\g_t-\tilde\g_t\|_2^2 + \|\g_t\|_2^2], \\
    &\leq \tfrac12\eta_t[\|\g_{t+1}-\hat\g_t\|_2^2 - \|\g_{t+1}\|_2^2 - \|\tilde\g_t-\hat\g_t\|_2^2 + \|\g_t\|_2^2], 
  \end{align*}
  where the third equality follows from the identity $\langle\a,\b\rangle=-\frac12[\|\a-\b\|_2^2-\|\a\|_2^2-\|\b\|_2^2]$,   the last equality follows from reorganization of the terms, and the last inequality follows from $\|\g_t-\tilde\g_t\|_2^2\geq0$. Therefore, as $\eta_t>0$, 
  \begin{align*}
    \|\g_{t+1}\|_2^2 - \|\g_t\|_2^2
    &\leq \|\g_{t+1}-\hat\g_t\|_2^2 - \|\tilde\g_t - \hat\g_t\|_2^2 \\
    &= \|F(\z_{t+1}) - F(\w_t)\|_2^2 - \left\|\tfrac{1}{\eta_t}(\z_{t+1}-\w_t)\right\|_2^2 \\
    &\leq \hat L_t^2\|\z_{t+1}-\w_t\|_2^2 - \tfrac{1}{\eta_t^2}\|\z_{t+1}-\w_t\|_2^2 \\
    &= \tfrac{1}{\eta_t^2}(\eta_t^2\hat L_t^2-1)\|\z_{t+1}-\w_t\|_2^2 \leq 0, 
  \end{align*}
  where the second inequality follows from the definition of $\hat L_t${, and the final conclusion follows from the premise that $\eta_t\hat L_t\leq1$.} 
\end{proof}

With the per-iteration analysis, we are now equipped to analyze specific stepsize schemes. In the following subsections, we apply these general results to provide last-iterate convergence rates for \cref{alg:extragradient-adaptive,alg:extragradient-backtracking}, which are designed to handle globally and locally Lipschitz continuous operators respectively. 

As we will see later, a key idea in developing adaptive stepsizes is to ensure the conditions $\eta_tL_t<1$ and $\eta_t\hat L_t\leq1$ required by \cref{lem:residual-bound,lem:residual-nonincreasing}, so that the extragradient residual is properly upper-bounded and ensuring that we make sufficient progress towards zero.

\subsection{{Lipschitz continuous $F$}}
\label{sec:alg-smooth}

In this subsection, we study the case where the operator $F$ satisfies a global {Lipschitz continuity} condition. Formally, we make the following assumption. 

\begin{assumption}\label{assum:smooth}
$F$ is $L$-Lipschitz, i.e., there exists a constant $L>0$ such that $\|F(\z')-F(\z)\|_2\leq L\|\z'-\z\|_2$ for all $\z,\z'\in\cZ$. 
\end{assumption}

Under \cref{assum:smooth}, a fixed stepsize $\eta_t < 1/L$ would suffice to ensure convergence. However, to overcome the challenge that the global Lipschitz constant $L$ is typically unknown or difficult to estimate in practice, we propose \cref{alg:extragradient-adaptive}, an adaptive variant of the extragradient algorithm, which dynamically adjusts its stepsize $\eta_t$ based on the local geometry of $F$.

\begin{algorithm}[htbp]
  \caption{Parameter-free non-ergodic extragradient (PF-NE-EG) algorithm} 
  \label{alg:extragradient-adaptive}
  \begin{algorithmic}
    \Require Initial solution $\z_0\in\cZ$, initial stepsize $\eta_0>0$, $\theta\in(0,1)$, and $\lambda_t\geq1$ \stt $\lim_{t\to\infty}\lambda_t=1$. 
    \For{$t=1,2,\dots,T$}
      \State Step 1. Stepsize selection: Compute $L_{t-1}$ and $\hat L_{t-1}$ according to \eqref{eq:local-Lipschitz-estimate-1}--\eqref{eq:local-Lipschitz-estimate-2}, and set 
      \begin{align*}
        \eta_t=\min\set{\lambda_{t-1}\eta_{t-1}, \frac{\theta}{L_{t-1}}, \frac{\theta}{\hat L_{t-1}}}, 
      \end{align*}
      unless $\w_{t-1}=\z_{t-1}$, in which case we stop and return $\z_{t-1}$. 
      \State Step 2. Extragradient update:
        \begin{align*}
          \w_t &= \Proj_{\cZ}(\z_t-\eta_tF(\z_t)), \\
          \z_{t+1} &= \Proj_{\cZ}(\z_t-\eta_tF(\w_t)). 
        \end{align*}
    \EndFor
    \Ensure $\z_T$.
  \end{algorithmic}
\end{algorithm}

As noted in \cref{rem:L-def}, $L_t$ is well-defined unless the optimality is reached, at which point the algorithm may terminate. 

The main ingredients in designing the stepsize $\eta_t$ in \cref{alg:extragradient-adaptive} are $1/L_{t-1}$ and $1/\hat L_{t-1}$, the inverses of the local Lipschitz estimates. This is in contrast to the threshold $1/L$ used in classical settings. We further scale $1/L_{t-1}$ by a factor $\theta \in (0,1)$ to ensure $\eta_{t} L_{t-1} < 1$ and $\eta_{t}\hat L_{t-1}\leq1$, which will be essential for the application of \cref{lem:residual-bound,lem:residual-nonincreasing}. 

To prevent erratic behaviors of $\eta_t$ between iterations, we impose $\eta_t \leq \lambda_{t-1}\eta_{t-1}$, where $\{\lambda_t\}_{t \in \mathbb{N}}$ is a sequence of positive numbers converging to $1$. When $\lambda_t = 1$, this yields a sequence of non-increasing stepsizes. By contrast, choosing $\lambda_t > 1$ (e.g., $\lambda_t=1+1/(1+\log(t+1))$) allows $\eta_t$ to increase and better adapt to the local geometry. This distinguishes our approach from most existing adaptive schemes, and it greatly alleviates the impact of poor initializations. Instead of being trapped at a vanishingly small value by an initially large $L_t$ or small $\eta_0$, \cref{alg:extragradient-adaptive} can dynamically adjust the stepsize as the landscape flattens. 

Equipped with this adaptive stepsize scheme, the iterate update of \cref{alg:extragradient-adaptive} is identical to the well-established extragradient.

The stepsize update involves only the computations of $L_t$ and $\hat L_t$ via \eqref{eq:local-Lipschitz-estimate-1}--\eqref{eq:local-Lipschitz-estimate-2} plus some basic operations. In particular, no additional operator evaluations beyond those already used in the extragradient steps are needed, matching the oracle complexity of the non-adaptive algorithm. Moreover, because the stepsize is determined by direct computation rather than an iterative line search, its computational overhead is kept minimal. 

The following result establishes the convergence rate of \cref{alg:extragradient-adaptive}. Specifically, we show that  it achieves $o(1/\sqrt{T})$ last-iterate convergence in the extragradient residual $\|F(\z_T)+\bm\xi_T\|_2$ without requiring any prior knowledge of the Lipschitz constant $L$.

\begin{proposition}\label{prop:extragradient-adaptive}
  Under \cref{assum:smooth}, 
  \cref{alg:extragradient-adaptive} achieves an $o(1/\sqrt{T})$ convergence in the extragradient residual $\|F(\z_T) + \bm\xi_T\|_2$. 
\end{proposition}

\begin{proof}
  Since $L_{t-1},\hat L_{t-1}\leq L$ by definition, the stepsize selection in \cref{alg:extragradient-adaptive} ensures that $\eta_t\geq\min\set{\eta_0,\frac{\theta}{L}}>0$ for all $t\in\N$. This implies $\liminf_{t\to\infty}\eta_{t+1}/\eta_{t}\geq1$. On the other hand, $\eta_{t+1}\leq\lambda_t\eta_t$ and $\lim_{t\to\infty}\lambda_t=1$ implies $\limsup_{t\to\infty}\eta_{t+1}/\eta_{t}\leq\lim_{t\to\infty}\lambda_t=1$. Thus, we have shown that $\lim_{t\to\infty}\eta_{t+1}/\eta_{t}=1$. 
  Since $\eta_{t+1}\leq\frac{\theta}{L_t}$ by the stepsize selection, we have $\eta_tL_t\leq\theta\frac{\eta_t}{\eta_{t+1}}\to\theta$ as $t\to\infty$. 
  Note that $\theta\in(0,1)$ implies $\theta<\frac{\theta+1}{2}<1$. Therefore, we have 
  $\eta_tL_t\leq\frac{\theta+1}{2}$ for $t$ sufficiently large, and similarly, $\eta_t\hat L_t\leq\frac{\theta+1}{2}$ for $t$ sufficiently large. Let $t_0$ be such that both inequalities hold for all $t\geq t_0-1$. 
  By \cref{lem:residual-bound}, for $t\geq t_0$, we have 
  \begin{align*}
    \eta_{t-1}^2\|F(\z_t) + \bm\xi_t\|_2^2 
    &\leq \frac{3+2\eta_{t-1}^2\hat L_{t-1}^2}{1-\eta_{t-1}L_{t-1}}[\|\z_{t-1}-\z_*\|_2^2 - \|\z_t-\z_*\|_2^2] \\
    &\leq \frac{6 + (\theta+1)^2}{1-\theta}[\|\z_{t-1}-\z_*\|_2^2 - \|\z_t-\z_*\|_2^2]. 
  \end{align*}
  Summing up the above inequality for $t=t_0,\dots,T$ and noting that the right-hand side telescopes, we have
  \begin{align*}
    \sum_{t=t_0}^{T}\eta_{t-1}^2\|F(\z_t) + \bm\xi_t\|_2^2 \leq \frac{6 + (\theta+1)^2}{1-\theta}\|\z_{t_0-1}-\z_*\|_2^2 < +\infty.
  \end{align*}
  {By taking the limit of both sides as $T\to\infty$, we get  $\sum_{t\ge t_0}\eta_{t-1}^2\|F(\z_t) + \bm\xi_t\|_2^2 <+\infty$. }
  Recall that $\eta_t\geq\min\set{\eta_0,\frac{\theta}{L_{t-1}},\frac{\theta}{\hat{L}_{t-1}}} \ge \min\set{\eta_0,\frac{\theta}{L}}:= \tilde{\eta} >0$ for all $t$. Thus, $\sum_{t\geq t_0}\|F(\z_t)+\bm\xi_t\|_2^2 \leq \frac{1}{\tilde{\eta}^2}  \sum_{t\ge t_0}\eta_{t-1}^2\|F(\z_t) + \bm\xi_t\|_2^2  < +\infty$. Then, from \cref{lem:residual-nonincreasing,lem:olivier}, we conclude that $\|F(\z_T) + \bm\xi_T\|_2^2 = o(1/T)$. 
\end{proof}

The proof of \cref{prop:extragradient-adaptive} shows that, while the stepsize update in \cref{alg:extragradient-adaptive} explicitly only enforces $\eta_{t} L_{t-1} < 1$ and $\eta_{t}\hat L_{t-1}\leq1$, after a sufficient number of iterations, $\eta_t L_t < 1$ and $\eta_t\hat L_t\leq1$ are eventually satisfied, which leads to an $o(1/\sqrt{T})$ rate of convergence in terms of the extragradient residual.

\subsection{Local Lipschitzness: non-monotone backtracking line search}
\label{sec:alg-backtracking-new}

While \cref{alg:extragradient-adaptive} achieves convergence under the assumption that $F$ is Lipschitz continuous, many practical problems, such as those with exponential growth over an unbounded domain, satisfy only a local {Lipschitz continuity} condition for $F$. To address such cases, in this section we work with the following {less stringent} assumption. 

\begin{assumption}\label{assum:local-smooth}
$F$ is locally Lipschitz continuous for all $\z\in\cZ$. That is, for any $\z\in\cZ$, there exists $L(\z)>0$ and a neighborhood $\N(\z)\subset\cZ$ of $\z$ such that 
\begin{align*}
  \|F(\z')-F(\z'')\|_2 \leq L(\z) \|\z'-\z''\|_2, \qquad\forall\z',\z''\in\N(\z). 
\end{align*}
\end{assumption}

To circumvent the lack of a global Lipschitz constant, we incorporate a backtracking line search procedure into \cref{alg:extragradient-adaptive} {while also allowing for non-monotone stepsizes}; see \cref{alg:extragradient-backtracking} for the formal description. 
\begin{algorithm}[htbp]
  \caption{Parameter-free non-ergodic extragradient (PF-NE-EG) algorithm with {non-monotone} backtracking line search }
  \label{alg:extragradient-backtracking}
  \begin{algorithmic}
    \Require Initial solution $\z_0\in\cZ$, initial stepsize $\eta_0>0$, $\theta\in(0,1)$, $\rho\in(0,1)$, and $\lambda_t\geq1$ \stt $\lim_{t\to\infty}\lambda_t=1$. 
    \For{$t=1,2,\dots,T$}
      \State {Step 1. Stepsize initialization: Compute $L_{t-1}$ and $\hat L_{t-1}$ according to \eqref{eq:local-Lipschitz-estimate-1}--\eqref{eq:local-Lipschitz-estimate-2}, and set 
      \begin{align*}
        \bar\eta_t = \min\set{\lambda_{t-1}\eta_{t-1}, \frac{\theta}{L_{t-1}}, \frac{\theta}{\hat L_{t-1}}}, 
      \end{align*}}
      \State Step 2. Backtracking: Starting with $\eta=\bar\eta_t$, decrease it by a factor of $\rho$ iteratively until it satisfies the conditions
      \begin{align}\label{eq:backtracking}
        \eta\frac{\|F(\w_t(\eta))-F(\z_t)\|_2}{\|\w_t(\eta)-\z_t\|_2} &\leq {\frac{\theta+1}{2}} < 1, \\
        \eta\frac{\|F(\w_t(\eta))-F(\z_{t+1}(\eta))\|_2}{\|\w_t(\eta)-\z_{t+1}(\eta)\|_2} &\leq 1,\quad\text{if }\w_t(\eta)\neq\z_{t+1}(\eta), \label{eq:backtracking-2}
      \end{align}
      where $\w_t(\eta)\coloneqq\Proj_{\cZ}(\z_t-\eta F(\z_t))$ and $\z_{t+1}(\eta)\coloneqq\Proj_{\cZ}(\z_t-\eta F(\w_t(\eta)))$, unless $\w_{t}(\eta)=\z_{t}$, in which case we stop and return $\z_t$. Set  $\eta_t=\eta$. 
      \State Step 3. Extragradient update:
        \begin{align*}
          \w_t &= \Proj_{\cZ}(\z_t-\eta_tF(\z_t)), \\
          \z_{t+1} &= \Proj_{\cZ}(\z_t-\eta_tF(\w_t)).
        \end{align*}
    \EndFor
    \Ensure $\z_T$.
  \end{algorithmic}
\end{algorithm}

Recall from \cref{rem:L-def} that we may assume $\|\w_t(\eta)-\z_t\|_2\neq0$ in \eqref{eq:backtracking} unless optimality is reached, in which case we simply terminate and return {$\z_t$}. On the other hand, if $\w_t(\eta)=\z_{t+1}(\eta)$, the condition \eqref{eq:backtracking-2} is viewed as automatically satisfied, and we have $\hat L_t=0$ by definition \eqref{eq:local-Lipschitz-estimate-2}. 

The main component of \cref{alg:extragradient-backtracking} is the backtracking line search, which is designed to directly satisfy the conditions $\eta_tL_t<1$ and $\eta_t\hat L_t\leq1$ of \cref{lem:residual-bound,lem:residual-nonincreasing} in every iteration. The factor $\theta \in (0, 1)$ again ensures that $\eta_tL_t$ is well separated from $1$, which is crucial for the convergence analysis. 

Although the backtracking line search procedure in \cref{alg:extragradient-backtracking} is a natural extension of \cref{alg:extragradient-adaptive}, it is worth noting that such an extension is not readily applicable to many existing adaptive regimes. In particular, aggregation-type stepsize regimes (e.g., \citep{antonakopoulos_extra-gradient_2024}), which rely on the aggregation based on all historical iterates, do not lend themselves easily to backtracking conditions. In these algorithms, the stepsize is a non-increasing function of the entire history, making it difficult to cope with local Lipschitz continuity.

While the line search introduces additional operator evaluations, the total computational overhead remains well-controlled. As we will see in \cref{lem:finite-backtracking}, the number of failed backtracking steps is finite. As such, it  adds only a constant to the overall complexity of the algorithm. 
Before presenting the main convergence result, we first ensure in the following lemma that the backtracking line search is well-defined.

\begin{lemma}\label{lem:backtracking-stops}
  Suppose \cref{assum:local-smooth} holds. At each iteration, the backtracking procedure in \cref{alg:extragradient-backtracking} stops within finitely many operations. Thus, $\eta_t>0$ holds for all $t$ until termination. 
\end{lemma}

\begin{proof}
    By local Lipschitz continuity of $F$, there exists a neighborhood $\N(\z_t)$ of $\z_t$ such that 
  \begin{align*}
    \|F(\z)-F(\z')\|_2 \leq L(\z_t) \|\z-\z'\|_2
  \end{align*}
  for any $\z,\z'\in\N(\z_t)$. 
Recall $\w_t(\eta)=\Proj_{\cZ}(\z_t-\eta F(\z_t))$, then as $z_t\in\cZ$ and so $\Proj_{\cZ}(\z_t)=\z_t$, using the nonexpansiveness of projection operation (see \cite[A.(3.1.6)]{hiriart-urruty_fundamentals_2001}), we get 
  \begin{align}
  \|\w_t(\eta)-\z_t\|_2 &= \left\| \Proj_{\cZ}(\z_t-\eta F(\z_t)) - \Proj_{\cZ}(\z_t) \right\|_2  \notag \\
  &\leq \left\| (\z_t-\eta F(\z_t)) - \z_t \right\|_2 = \eta\|F(\z_t)\|_2. \label{eq:local-smooth-distance-bound}
  \end{align} 
  Hence, for sufficiently small $\eta>0$, we can guarantee $\w_t(\eta)\in\N(\z_t)$ and $\eta L(\z_t)\leq\theta$. Then, for such a sufficiently small $\eta>0$, using the local Lipschitz continuity of $F$,  we have 
  \begin{align*}
    \eta\frac{\|F(\w_t(\eta))-F(\z_t)\|_2}{\|\w_t(\eta)-\z_t\|_2} 
    \leq \eta L(\z_t) \leq \theta \leq \frac{\theta+1}{2}<1 ,
  \end{align*}
  {where the last two inequalities follow from $\theta\in(0,1)$.}   In addition, using the definition of $\w_t(\eta)$ and $\z_{t+1}(\eta)$ and nonexpansiveness of the projection, we have 
  \begin{align*}
  \|\w_t(\eta)-\z_{t+1}(\eta)\|_2 &\leq \eta\|F(\z_t)-F(\w_t(\eta))\|_2 \\
  & \leq\eta L(\z_t)\|\z_t-\w_t(\eta)\|_2\leq\eta^2L(\z_t)\|F(\z_t)\|_2, 
  \end{align*}
  where the second inequality follows from $\w_t(\eta)\in\N(\z_t)$ and the local Lipschitz continuity of $F$, and
   the last inequality follows from \eqref{eq:local-smooth-distance-bound}.
  Thus, 
  \begin{align*}
  \|\z_{t+1}(\eta)-\z_t\|_2 &\leq \|\z_{t+1}(\eta)-\w_t(\eta)\|_2 + \|\w_t(\eta)-\z_t\|_2 \\
  &\leq \eta^2L(\z_t)\|F(\z_t)\|_2 + \eta\|F(\z_t)\|_2.
  \end{align*} 
  If $\w_t(\eta)\neq\z_{t+1}(\eta)$,
   then when $\eta>0$ is sufficiently small, $\w_t(\eta),\z_{t+1}(\eta)\in\N(\z_t)$ and $\eta L(\z_t)\leq1$, and 
  \begin{align*}
    \eta\frac{\|F(\w_t(\eta))-F(\z_{t+1}(\eta))\|_2}{\|\w_t(\eta)-\z_{t+1}(\eta)\|_2}
    \leq \eta L(\z_t) \leq 1. 
  \end{align*}
  Therefore, at iteration $t$, \eqref{eq:backtracking} and \eqref{eq:backtracking-2} can be satisfied when $\eta>0$ is sufficiently small, i.e., after $\eta$ is decreased finitely many times. 
\end{proof}

Next, we characterize the overall additional cost incurred by the line search procedure. The following lemma shows that the total number of additional operator evaluations during the line search procedure is finite. In other words, the algorithm eventually performs only two operator evaluations per iteration, thereby matching the oracle complexity of the standard extragradient method.

\begin{lemma}\label{lem:finite-backtracking}
  Suppose \cref{assum:local-smooth} holds. The backtracking line search procedure in \cref{alg:extragradient-backtracking} stops decreasing the stepsize within finitely many operations throughout all the iterations. In particular, there exists $\bar\eta>0$ such that $\eta_t\geq\bar\eta$ for all $t\in\N$. 
\end{lemma}

\begin{proof}  
  From \cref{lem:descent}, we have 
  \begin{align*}
    \|\z_{t+1} - \z_*\|_2^2 
    &\leq \|\z_t - \z_*\|_2^2 - (1 - \eta_tL_t) \left[\|\z_t-\w_t\|_2^2 + \|\z_{t+1} - \w_t\|_2^2 \right]. 
  \end{align*}
  By our line search procedure, we have $\eta_tL_t\leq\frac{\theta+1}{2}<1$, thus $\|\z_t-\z_*\|_2 \leq \|\z_0-\z_*\|_2$, which shows $\{\z_t\}_{t\in\N}$ is bounded. 
  {Since $1-\eta_tL_t\geq\frac{1-\theta}{2}>0$, we also have $\|\z_t-\w_t\|_2^2 + \|\z_{t+1} - \w_t\|_2^2\leq\frac{2}{1-\theta}[\|\z_t-\z_*\|_2^2 - \|\z_{t+1}-\z_*\|_2^2]$. Therefore, $\|\z_t-\w_t\|_2^2 \leq\frac{2}{1-\theta}\|\z_t-\z_*\|_2^2$, and 
  \begin{align*}
    \|\w_t-\z_*\|_2
    &\leq \|\w_t-\z_t\|_2 + \|\z_t-\z_*\|_2 \\
    &\leq \sqrt{\tfrac{2}{1-\theta}}\|\z_t-\z_*\|_2 + \|\z_t-\z_*\|_2 \\
    &\leq \left(\sqrt{\tfrac{2}{1-\theta}} + 1\right) \|\z_0-\z_*\|_2. 
  \end{align*}
  This shows that $\{\w_t\}_{t\in\N}$ is also bounded. The local Lipschitz continuity of $F$ implies its Lipschitz continuity on any compact set (see \cite[Theorem 2.1.6]{cobzas_lipschitz_2019}). 
  Thus, there exists a Lipschitz constant $L(\z_0,\z_*)>0$ \stt
  \begin{align*}
    \|F(\z')-F(\z)\|_2 \leq L(\z_0,\z_*) \|\z'-\z\|_2. 
  \end{align*}
  for any $\z,\z'\in\cZ$ satisfying $\|\z-\z_*\|_2,\|\z'-\z_*\|_2\leq\left(\sqrt{\tfrac{2}{1-\theta}} + 1\right) \|\z_0-\z_*\|_2$. Then, by definition, we have $L_{t-1},\hat L_{t-1}\leq L(\z_0,\z_*)$.}
  The stepsize selection in \cref{alg:extragradient-backtracking} ensures that $\bar\eta_t\geq\min\set{\eta_0,\frac{\theta}{L(\z_0,\z_*)}}>0$ for all $t\in\N$. This implies $\liminf_{t\to\infty}\bar\eta_{t+1}/\bar\eta_{t}\geq1$. On the other hand, {by definition of $\bar{\eta}_{t+1}$, we have} $\bar\eta_{t+1}\leq\lambda_t\eta_t\leq\lambda_t\bar\eta_t$ {(since $\eta_t\le \bar\eta_t$ for all $t$). Then, together with} $\lim_{t\to\infty}\lambda_t=1$ this implies $\limsup_{t\to\infty}\bar\eta_{t+1}/\bar\eta_{t}\leq\lim_{t\to\infty}\lambda_t=1$. Thus, we have shown that $\lim_{t\to\infty}\bar\eta_{t+1}/\bar\eta_{t}=1$. 
  Since $\bar\eta_{t+1}\leq\frac{\theta}{L_t}$ by the stepsize selection, we have $\bar\eta_tL_t\leq\theta\frac{\bar\eta_t}{\bar\eta_{t+1}}\to\theta$ as $t\to\infty$. 
  Note that $\theta\in(0,1)$ implies $\theta<\frac{\theta+1}{2}<1$. Therefore, we have 
  $\bar\eta_tL_t\leq\frac{\theta+1}{2}$ for sufficiently large $t$, and similarly, $\bar\eta_t\hat L_t\leq\frac{\theta+1}{2}$ for sufficiently large $t$ as well. Let $t_0$ be such that both inequalities hold for $t\geq t_0$. 
  Then, for $t\geq t_0$, \eqref{eq:backtracking}--\eqref{eq:backtracking-2} are satisfied by $\eta=\bar\eta_t$. Therefore, no backtracking step is needed thereafter. Noting from \cref{lem:backtracking-stops} that each iteration performs finitely many backtracking steps, we conclude that the total number of backtracking steps in \cref{alg:extragradient-backtracking} is finite. 
\end{proof}

Finally, we are now ready to state the convergence rate for the locally Lipschitz setting. As shown below, \cref{alg:extragradient-backtracking} achieves a rate of $o(1/\sqrt{T})$, i.e., the same order as \cref{alg:extragradient-adaptive}. In fact, by explicitly enforcing the conditions $\eta_{t} L_{t} < 1$ and $\eta_{t}\hat L_{t}\leq1$ via backtracking, the resulting bounds of  \cref{alg:extragradient-backtracking} are slightly better in terms of the constants, with the tradeoff being a constant number of additional line search steps.

\begin{proposition}\label{prop:extragradient-backtracking}
  Suppose \cref{assum:local-smooth} holds.   Then,   \cref{alg:extragradient-backtracking} achieves an $o(1/\sqrt{T})$ convergence in the extragradient residual $\|F(\z_T) + \bm\xi_T\|_2$. 
\end{proposition}

\begin{proof}
  Since the backtracking line search step in \cref{alg:extragradient-backtracking} guarantees that $\eta_tL_t\leq\tfrac{\theta+1}{2}$ and $\eta_t\hat L_t\leq1$ for all $t\in\N$, 
  by \cref{lem:residual-bound},  
  \begin{align*}
    \eta_t^2\|F(\z_t) + \bm\xi_t\|_2^2 
    &\leq \frac{3+2\eta_{t-1}^2\hat L_{t-1}^2}{1-\eta_{t-1}L_{t-1}} \left[\|\z_{t-1}-\z_*\|_2^2 - \|\z_t-\z_*\|_2^2 \right] \\
    &\leq \frac{{10}}{1-\theta} \left[ \|\z_{t-1}-\z_*\|_2^2 - \|\z_t-\z_*\|_2^2 \right]. 
  \end{align*}
  Summing up the above inequality for $t=1,\dots,T$ and noting that the right-hand side telescopes, we have
  \begin{align*}
    \sum_{t\in[T]}\eta_t^2\|F(\z_t) + \bm\xi_t\|_2^2 \leq \frac{6 + (\theta+1)^2}{1-\theta}\|\z_{0}-\z_*\|_2^2. 
  \end{align*}
  Recall from \cref{lem:finite-backtracking} that $\eta_t\geq\bar\eta>0$ for all $t\in\N$. Thus, $\sum_{t\geq1}\|F(\z_t)+\bm\xi_t\|_2^2 < +\infty$. By \cref{lem:residual-nonincreasing,lem:olivier}, we conclude that $\|F(\z_T) + \bm\xi_T\|_2^2 = o(1/T)$. 
\end{proof}

While \cref{alg:extragradient-backtracking} provides a robust approach for non-monotone adaptive stepsizes under local Lipschitz continuity, in \cref{alg:extragradient-backtracking-old} we also consider a monotone variant by using standard backtracking line search. This variant serves as a baseline for our numerical experiments. A detailed description of \cref{alg:extragradient-backtracking-old}, along with its convergence analysis and a stepsize increase trick used in our experiments, is provided in \cref{sec:alg-backtracking-old}.

\section{Numerical Results}
\label{sec:numerical}

{We test our algorithms on four different problem classes:  bilinear matrix game, LASSO problem, a group fairness classification problem, and a state-of-the-art relaxation for the maximum entropy sampling problem.} 
All experiments are coded in Python 3.9 and ran on a Linux server with a 3-GHz Intel Xeon Gold 5317 processor with 12 cores and 128 GB of RAM. The code {for the implementation of the algorithms tested} 
is available at \url{https://github.com/joyshen07/pf-ne-eg}. 

In addition to our proposed algorithms, we also implement and {compare} the standard extragradient algorithm (\texttt{EG}) with fixed stepsize, as well as \texttt{Universal MP} \citep{bach_universal_2019}, \texttt{Adaptive MP} \citep{antonakopoulos_adaptive_2019}, \texttt{AdaProx} \citep{antonakopoulos_adaptive_2021}, \texttt{AdaPEG} \citep{ene_adaptive_2022}, \texttt{aGRAAL} \citep{malitsky_golden_2020}, and \texttt{Adapt EG} \citep{antonakopoulos_extra-gradient_2024} whenever applicable (see \cref{tab:adaptive-literature}). 
To save space, we denote \cref{alg:extragradient-backtracking} as \texttt{PF-NE-EG AdaBt}, and \cref{alg:extragradient-backtracking-old} as \texttt{PF-NE-EG Bt}, where ``Bt'' stands for backtracking.
Throughout the experiments, we set $\lambda_t=1+\frac1{\log(t+2)}$ for \cref{alg:extragradient-adaptive}, $\rho=0.9$ for \cref{alg:extragradient-backtracking,alg:extragradient-backtracking-old}, and $\theta=0.9$ for {all variants of \texttt{PF-NE-EG}}. We adopt the stepsize increase trick for \cref{alg:extragradient-backtracking-old} mentioned in \cref{rem:stepsize-increase-trick} to ensure robustness. In addition, we take $D$ and $G_0$ for \texttt{Universal MP} exactly according to the best choice suggested by \cite{bach_universal_2019}, $\theta=0.9$ for \texttt{Adaptive MP}, $\phi=\frac{\sqrt{5}+1}{2}$ and $\bar\lambda=\eta_0$ to be the initial stepsize for \texttt{aGRAAL}, and $\eta=1$ for \texttt{AdaPEG}. All algorithms are initialized with the same stepsize and initial points in each experiment, with the exception of  standard \texttt{EG}. For problem instances where the Lipschitz constant $L$ is tractable, we set the stepsize for standard \texttt{EG} to $\eta=0.9/L<1/L$; for those where $L$ is unknown, we set the stepsize of  standard \texttt{EG} to be half the stepsize used for the adaptive algorithms, to compensate for its lack of adaptivity. Note that there is no theoretical guarantees supporting \texttt{EG} in the latter case, and we adopt the heuristic stepsize solely for experimental purposes. See the subsections below for further details specific to each problem class.

For the convergence plots, we run the algorithms for a fixed number of iteration, or until a target precision of $10^{-6}$ is reached, whichever comes first. In the solution time tables, we report the solution time (seconds) or iteration counts of the algorithms that reach a target precision $\varepsilon$ within the predefined runtime limit.

\subsection{Bilinear matrix game}
\label{sec:minimax-game}

In this subsection, we study the matrix game problem given by (see \citep{nemirovski_prox-method_2004}) 
\begin{align*}
  \min_{\x\in\Delta_d}\max_{\y\in\Delta_d}\x^\top\bA\y, 
\end{align*}
where $\Delta_d$ is the standard simplex in $\R^d$. 
Matrix games are a standard benchmark for evaluating algorithms for convex-concave SPP, especially extragradient-type methods. The problem corresponds to a zero-sum game in which players choose strategies $x$ and $y$ from the probability simplex, and the goal is to compute a Nash equilibrium by solving the bilinear min-max formulation above. This problem is well suited as a starting point for comparing the performance of SPP algorithms due to the simplicity of the bilinear form and the simplex domain. 

\paragraph{{Problem data.}}
Following the experimental setup in \citep{nemirovski_prox-method_2004}, we consider square matrices $\bA\in\R^{d\times d}$, where each entry $A_{ij}$ is selected to be nonzero independently with a pre-specified probability $\kappa\in(0,1)$, then the values are sampled from the uniform distribution on $[-1,1]$ for the selected nonzero entries. We examine  three sets of instances: $(d, \kappa)=(100, 1.0), (500, 0.2)$, and $(1000, 0.1)$. 

\paragraph{Implementation details.} 
For the algorithms that require knowledge of problem parameters, we take the domain diameter $D:=\sqrt{2}$, and Lipschitz constant $L:=\|\bA\|_2$.
All algorithms are initialized at $\x=\y=\frac1d\ones$ with initial stepsize $\eta_0=0.5$ except for  standard \texttt{EG} with constant stepsize $\eta=0.9/L$. In addition, for the instance $(d,\kappa)=(100,1.0)$, we also vary the initial stepsize to a smaller value $\eta_0=0.02$ to investigate the algorithms' sensitivity to it, including for the standard \texttt{EG}. 
For matrix games, the saddle point gap at $\z=(\x,\y)$ can be computed easily by
\begin{align*}
  G_{\Delta_d\times\Delta_d}(\z) = \max_{i\in[d]}(\bA^\top\x)_i - \min_{i\in[d]}(\bA\y)_i . 
\end{align*}
{Thus, for this problem class, we report this convergence metric for all algorithms tested}.

\paragraph{Performance comparison.}
\cref{fig:matrix-game,tab:matrix-game} compare the convergence behaviors of different algorithms for solving the matrix game instances. 
{\cref{fig:matrix-game} demonstrates} that there is a clear distinction between the convergence of ergodic and last-iterate algorithms, with the latter achieving significantly higher precision. Moreover,  \cref{alg:extragradient-adaptive,alg:extragradient-backtracking,alg:extragradient-backtracking-old} maintain consistently stable and superior performance. In terms of the last-iterate algorithms, while standard \texttt{EG} and \texttt{Adapt EG} appear as the closest competitors to our algorithms, both are  subject to major limitations. The competitive results of standard \texttt{EG} are due to the use of an optimal stepsize computed from the actual Lipschitz constant, which is rarely available in practical applications. In contrast, our proposed algorithms match its performance without requiring any problem-specific constants. Furthermore, while \texttt{Adapt EG} occasionally reaches a target precision of $\varepsilon=10^{-5}$ faster, it suffers from a significantly slower initial phase and erratic, non-monotonic behavior. As shown in the $(d,\kappa)=(100, 1.0)$ instance, \texttt{Adapt EG} is highly sensitive to initial stepsize selection: a large initial stepsize leads to the largest initial error gap, while a small initial stepsize causes it to  {have a similar convergence behavior as the ergodic algorithms}. This dependency of empirical behavior of \texttt{Adapt EG} on carefully chosen initial stepsize contradicts the fundamental goal of parameter-free design, and is in contrast to the robustness of \texttt{PF-NE-EG} algorithms. 
{In addition, based on the computation times reported in \cref{tab:matrix-game}, we observe that algorithms with last-iterate guarantees are significantly faster, and \cref{alg:extragradient-adaptive,alg:extragradient-backtracking} being the best and often beating the performance of standard \texttt{EG} utilizing the optimum stepsize.} 

\begin{figure}[htbp]
  \centering
  \includegraphics[width=.49\textwidth, trim=.8cm .8cm 0 0 0, clip]{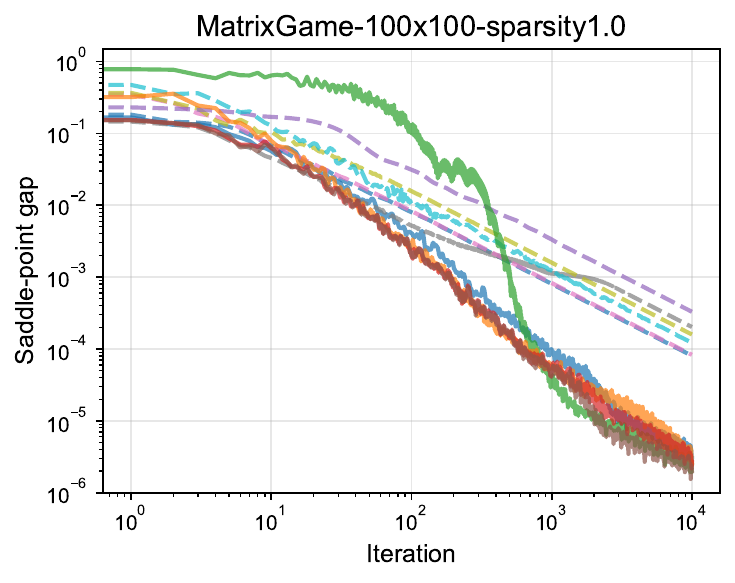}
  \includegraphics[width=.49\textwidth, trim=.8cm .8cm 0 0 0, clip]{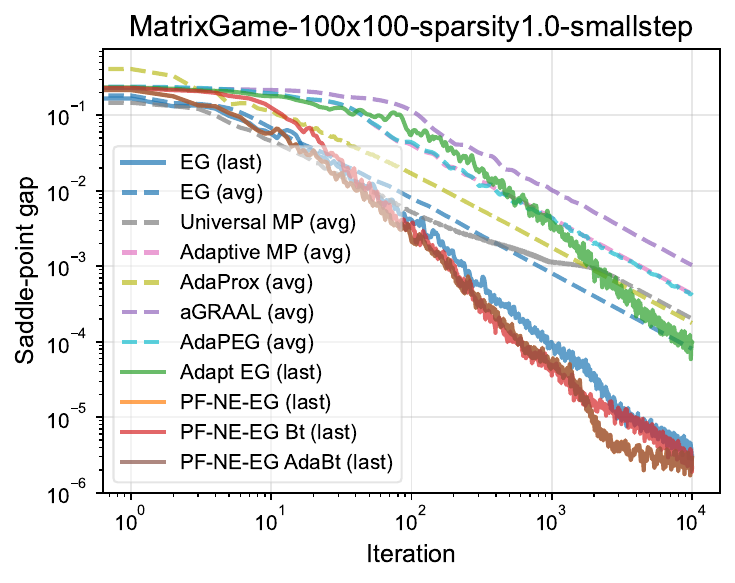}
  \includegraphics[width=.49\textwidth, trim=.8cm 0 0 0 0, clip]{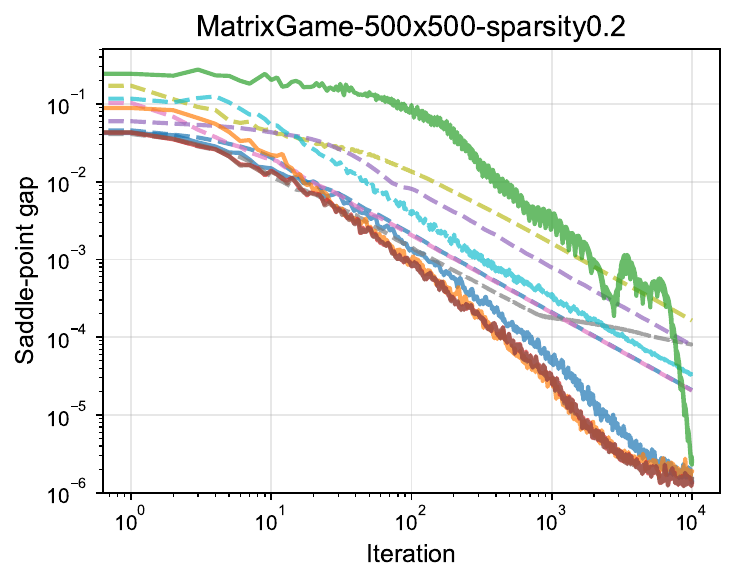}
  \includegraphics[width=.49\textwidth, trim=.8cm 0 0 0 0, clip]{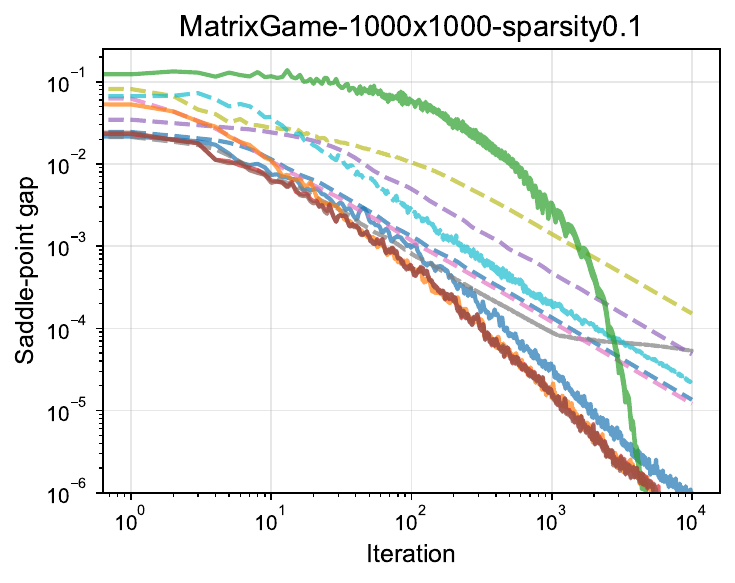}
  \caption{Convergence comparison of different algorithms on bilinear matrix game instances. Initial stepsizes are all set to $\eta_0=0.5$, except for the top right where $\eta_0=0.02$.     }
  \label{fig:matrix-game}
\end{figure}

\begin{table}[htbp]
\centering
\begin{tabular}{lSSSS}
\toprule
{Algorithm} & {(100, 1.0)} & {(100, 1.0)\textsuperscript{*}} & {(500, 0.2)} & {(1000, 0.1)} \\
\midrule
EG            & 0.24 & 9.01 & 0.88 & 0.63 \\
EG (avg)      & 6.13 & 32.60 & 6.95 & 5.03 \\
Universal MP  & 18.87 & 18.70 & 105.46 & 64.34 \\
Adaptive MP   & 7.88  & 40.35 & 8.45 & 5.09 \\
AdaProx       & 13.59 & 14.93 & 74.41 & 60.51 \\
aGRAAL        & 20.56 & 65.51 & 18.86 & 13.20 \\
AdaPEG        & 6.55  & 21.60 & 7.06 & 5.31 \\
Adapt EG      & 0.16  & 9.45  & 3.42 & 1.37 \\
PF-NE-EG      & 0.52  & 0.21  & 0.61 & 0.82 \\
PF-NE-EG Bt   & 0.61  & 0.61  & 1.31 & 1.08 \\
PF-NE-EG AdaBt& 0.23  & 0.22  & 0.68 & 0.55 \\
\bottomrule
\end{tabular}
\caption{Time (in seconds) to reach $\varepsilon = 10^{-5}$ for bilinear matrix game instances. Initial stepsizes are all set to $\eta_0=0.5$, except for the second column where $\eta_0=0.02$. }
\label{tab:matrix-game}
\end{table}

\subsection{LASSO}

Next, we consider the Least Absolute Shrinkage and Selection Operator (LASSO) problem, widely used in compressive sensing and high-dimensional statistics. As a nonsmooth convex minimization problem, it can be reformulated into a smooth convex-concave saddle point problem and used for testing saddle point algorithms \citep{liu_new_2026}. 

The LASSO problem is defined as
\[
\min_{\x \in \R^n} \frac{1}{2} \|\bA\x - \b\|_2^2 + \lambda \|\x\|_1
\]
where $\bA \in \R^{m \times n}$, $\b \in \R^m$, and $\lambda > 0$ is the regularization parameter. By the dual representation of the $\ell_1$-norm, i.e., $\lambda \|\x\|_1 = \max_{\|\y\|_\infty \leq \lambda} \langle \y, \x \rangle$, we have
\[
\min_{\x \in \R^n} \max_{\|\y\|_\infty \leq \lambda} \set{ \phi(\x,\y) \coloneq \frac{1}{2} \|\bA\x - \b\|_2^2 + \langle \y, \x \rangle }.
\]
Note that {while the primal domain is simply $\R^n$,} the dual domain is constrained, {bounded}, and easy to perform projection onto. 
The operator associated with $\phi(\x,\y)$ is
\[
F(\z) = \begin{pmatrix} \nabla_\x \phi(\x, \y) \\ -\nabla_\y \phi(\x, \y) \end{pmatrix}
= \begin{pmatrix} \bA^\top(\bA\x - \b) + \y \\ -\x \end{pmatrix},
\]
which is linear, and therefore Lipschitz continuous over the entire domain.

\paragraph{{Problem data.}}
We consider an underdetermined system where $\bA \in \R^{m \times n}$ with $m < n$. The matrix $\bA$ is generated by sampling entries from a standard normal distribution, followed by a column-normalization step such that $\|a_j\|_2 = 1$ for all $j=1, \dots, n$. This normalization ensures that the local curvature is not dominated by a single column's scale. The vector $\b$ is constructed as $\b = \bA\x_{\text{true}} + \bm\epsilon$ where $\bm\epsilon\in\R^m$ is an additive Gaussian noise with standard deviation $\sigma = 0.01$. The ground-truth $\x_{\text{true}}$ is generated to be $s$-sparse, with non-zero entries sampled from a Gaussian distribution. 
We take $(m, n, s) = (250, 1000, 0.5), (500, 5000, 0.1)$, and set the regularization parameter $\lambda = 1$ to {generate two} different instances. 

\paragraph{Implementation details.}
All algorithms are initialized at $\x=\y = \mathbf{0}$, with initial stepsize $\eta_0=0.1$, except for \texttt{EG} with $\eta=0.05$. Since the primal domain is unbounded, the \texttt{Universal MP} \citep{bach_universal_2019} does not apply and is not tested for LASSO. 
To compare the algorithm performance, {as the computation needed for saddle point gap is rather expensive for this problem, we instead} compute and monitor the natural residual $R_{0.01}(\z_t)$ for those with last-iterate convergence guarantees, or $R_{0.01}(\bar\z_t)$ for those with ergodic convergence, where $\bar\z_t$ denotes the (weighted) average of iterates. 

\paragraph{Performance comparison.}
Results are shown in \cref{fig:lasso,tab:lasso}. 
In this problem class, \cref{alg:extragradient-adaptive,alg:extragradient-backtracking,alg:extragradient-backtracking-old} exhibit a more pronounced advantage. They convergence significantly faster, reaching the target precision of $\varepsilon=10^{-6}$ in substantially fewer iterations as well as solution time than all other algorithms. The performance difference between last-iterate and ergodic algorithms remains clear, and the  last-iterate algorithms reach $\varepsilon=10^{-6}$ within $60$ seconds. Meanwhile, the gap between \texttt{PF-NE-EG} algorithms and their closest last-iterate competitors, standard \texttt{EG} and \texttt{Adapt EG}, has widened. Specifically, \cref{alg:extragradient-adaptive} reaches the precision threshold over {four} times faster than \texttt{Adapt EG} and over fourteen times faster than standard \texttt{EG} in terms of solution time, and the advantage is even more pronounced in the high-dimensional sparse instance. The additional computational cost of \cref{alg:extragradient-backtracking-old} compared to \cref{alg:extragradient-adaptive,alg:extragradient-backtracking} is mainly due to the use of the stepsize increase trick (see \cref{rem:stepsize-increase-trick}), which leads to the operator evaluations (gradient computations) to double. Even so, the solution time it takes to convergence remains competitive among {other algorithms from the literature}. These results validate that the theoretical efficiency of our proposed algorithms translates into significant practical gains.

\begin{figure}[htbp]
  \centering
  \includegraphics[width=.49\textwidth, trim=.8cm 0 0 0 0, clip]{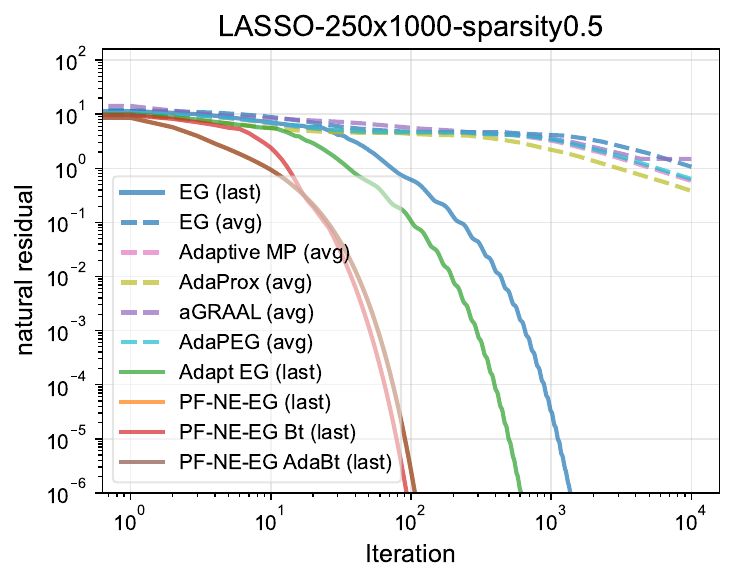}
  \includegraphics[width=.49\textwidth, trim=.8cm 0 0 0 0, clip]{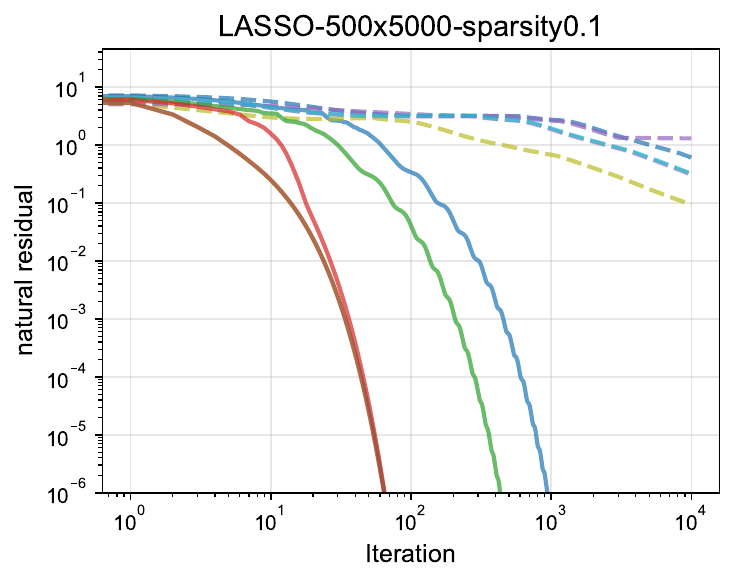}
  \caption{Convergence comparison of different algorithms for LASSO instances.}
  \label{fig:lasso}
\end{figure}

\begin{table}[htbp]
\centering
\begin{tabular}{lSS}
\toprule
{Algorithm} & {$(250,1000,0.5)$} & {$(500,5000,0.1)$} \\
\midrule
EG          & 0.42  & 1.88 \\
Adapt EG    & 0.13  & 0.51 \\
PF-NE-EG    & 0.03  & 0.10 \\
PF-NE-EG Bt & 0.06  & 0.21 \\
PF-NE-EG AdaBt& 0.03& 0.10 \\
\bottomrule
\end{tabular}
\caption{Time (in seconds) to reach $\varepsilon = 10^{-6}$ for LASSO instances.}
\label{tab:lasso}
\end{table}

\subsection{Group fairness classification}

In this part, we consider the problem of training a fair binary classifier across $m$ distinct demographic groups via the minimax fairness model \citep{diana_minimax_2021}. Modern machine learning models often achieve high overall accuracy at the expense of specific subgroups, leading to biased outcomes in sensitive domains like hiring or credit scoring. The minimax fairness model addresses this by minimizing the worst-case error across all groups, effectively prioritizing the most disadvantaged subpopulation.

Let $\{ (\bX_i, \y_i) \}_{i=1}^m$ denote the datasets for each group, where the features are represented by $\bX_i=(\x_{i1}^\top,\dots,\x_{in_i}^\top)^\top\in \R^{n_i \times d}$ and $\y_i \in \{0, 1\}^{n_i}$ are the labels. We seek to find a model parameter $\bm\theta \in \R^d$ that minimizes the maximum risk across all groups, formulated as the following saddle point problem
\[
  {\min_{\bm\theta \in \R^d}  \max_{i\in[m]} \set{\ell_i(\bm\theta)} =} \min_{\bm\theta \in \R^d} \max_{\q \in \Delta_m} \set{ \phi(\bm\theta, \q) \coloneq \sum_{i=1}^m q_i \ell_i(\bm\theta) }, 
\]
where $\ell_i(\bm\theta)$ represents the exponential loss for group $i$:
\[
  \ell_i(\bm\theta) = \frac{1}{n_i} \sum_{j=1}^{n_i} \exp(-y_{ij}\bm\theta^\top\x_{ij}).
\]
This is a convex-concave problem with nonlinear coupling between $\bm\theta$ and $\q$, and the primal domain is unbounded. The gradients are
\begin{align*}
  \nabla_{\bm\theta} \phi &= \sum_{i=1}^m q_i \nabla \ell_i(\bm\theta), \\ 
  \nabla_{\q} \phi &= [\ell_1(\bm\theta), \ell_2(\bm\theta), \dots, \ell_m(\bm\theta)]^\top,
\end{align*}
where
\[
\nabla \ell_i(\bm\theta) = \frac{1}{n_i} \sum_{j=1}^{n_i} \left(-y_{ij}\exp(-y_{ij}\bm\theta^\top\x_{ij})\right)\x_{ij}. 
\]
The monotone operator $F=(-\nabla_{\bm\theta}\phi,-\nabla_\q\phi)$ is locally Lipschitz continuous due to its analyticity. However, it is not globally Lipschitz continuous over the domain as it grows exponentially fast w.r.t.\ $\bm\theta$. 

\paragraph{{Problem data.}}
We generate a synthetic dataset with $m$ groups using the \sloppy  \texttt{make\_classification} function from the Python package \texttt{sklearn.datasets}. For each group $i$, we generate an equal number of $n_i=n$ samples with $d$ features . To simulate heterogeneous groups, we vary the prior probability {$\P(y=1)=0.5+0.1\cdot(i/m)$ for each group $i\in[m]$}. This forces the dual variable $\q$ to prioritize groups where the minority class is underrepresented and harder to classify. We also add group-specific label noise by {randomizing a percentage of $10\%\cdot(i/m)^2$} of labels in group $i$. 
We generated instances with $(m, n, d)=(10, 200, 100), (20, 200, 50)$. 

\paragraph{Implementation details.}
The primal variable is initialized to $\bm\theta=\mathbf{0}$, and the dual variable $\q$ is initialized as center of the simplex, $\q = \frac{1}{m}\ones$. The initial stepsize is set to $\eta_0=0.01$. We implement all the algorithms in \cref{tab:adaptive-literature} regardless of whether their assumptions are satisfied by the group fairness classification problem{. However, many algorithms encounter numerical overflow issues in the experiment, and we omit the results for those.}
To compare the algorithm performance, we compute and monitor the natural residual $R_{0.01}(\z_t)$ for those with last-iterate convergence guarantees, or $R_{0.01}(\bar\z_t)$ for those with ergodic convergence, where $\bar\z_t$ denotes the (weighted) average of iterates. 

\paragraph{Performance comparison.}
Results are shown in \cref{fig:fairness,tab:fairness}. 
These results highlight the robustness of \cref{alg:extragradient-adaptive,alg:extragradient-backtracking,alg:extragradient-backtracking-old} in highly nonlinear problems without global Lipschitz continuous operators. Notably,  \texttt{aGRAAL} is the only algorithm from the  literature that does not incur numerical overflow, which is in accordance with its theoretical guarantee under local Lipschitz assumption. However, its performance is substantially weaker than our proposed algorithms, and it fails to reach high precision within a practical timeframe. In contrast, \cref{alg:extragradient-adaptive,alg:extragradient-backtracking,alg:extragradient-backtracking-old} exhibit fast convergence behaviors after an initial phase, achieving a target precision of $\varepsilon=10^{-6}$ by roughly $10^4$ iterations. We also note that \cref{alg:extragradient-backtracking-old} again takes roughly twice as much time as \cref{alg:extragradient-adaptive}, as expected due to the stepsize increase trick (\cref{rem:stepsize-increase-trick}), whereas \cref{alg:extragradient-backtracking} achieves a comparable solution time.

\begin{figure}[htbp]
  \centering
  \includegraphics[width=.49\textwidth, trim=.8cm 0 0 0 0, clip]{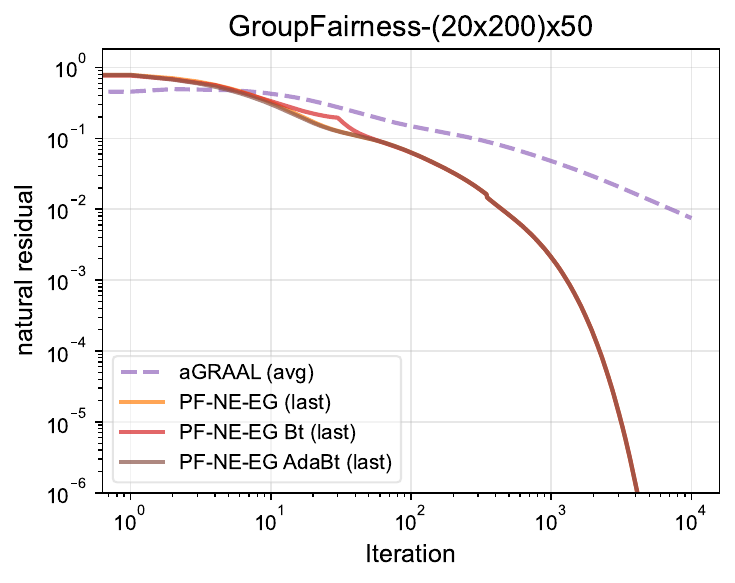}
  \includegraphics[width=.49\textwidth, trim=.8cm 0 0 0 0, clip]{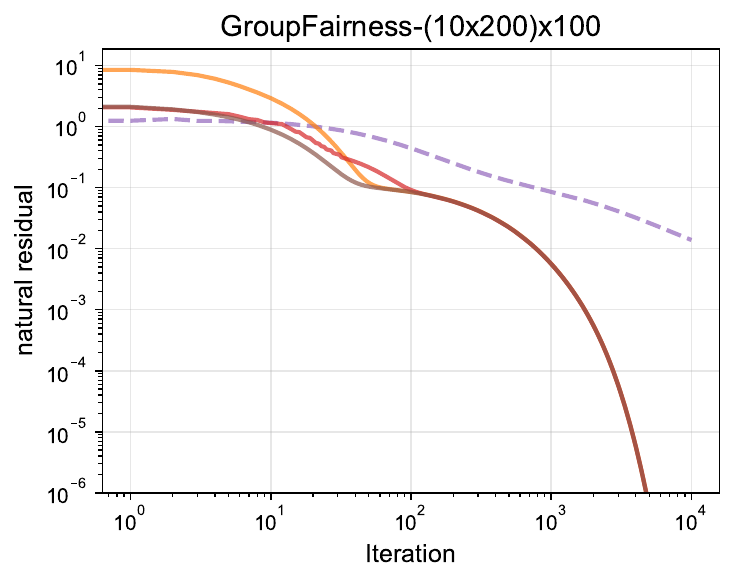}
  \caption{Convergence comparison of different algorithms for solving group fairness classification. }
  \label{fig:fairness}
\end{figure}

\begin{table}[htbp]
\centering
\begin{tabular}{lSS}
\toprule
Algorithm & {$(20,200,50)$} & {$(10,200,100)$} \\
\midrule
PF-NE-EG       & 5.01 & 3.62  \\
PF-NE-EG Bt    & 9.76 & 7.21 \\
PF-NE-EG AdaBt & 4.86 & 3.67  \\
\bottomrule
\end{tabular}
\caption{Time (in seconds) to reach $\varepsilon = 10^{-6}$ for group fairness classification.}
\label{tab:fairness}
\end{table}

\subsection{Maximum-Entropy Sampling Problem (MESP)}
\label{sec:mesp}

Finally, we test the effectiveness of the proposed algorithms on relaxations of the \emph{maximum-entropy sampling problem} (MESP), a well-known NP-hard problem arising in optimal experimental design. {See \cite{fampa_maximum-entropy_2022,fampa_recent_2026} for recent comprehensive treatments.}
Given a positive semidefinite covariance matrix $\bC \in \S^d_{+}$ and a subset size $s \leq \rank(\bC)$, MESP seeks a subset $S \subseteq [d]$ with $|S| = s$ that maximizes the information of the corresponding principal submatrix:
\[
  \max_{S \subseteq [d],\, |S| = s} \log\det(\bC_{S,S}).
\]
This criterion coincides with the classical D-optimal design objective, which aims to maximize the determinant of the information matrix associated with the selected variables.

A central tool for addressing the MESP is the use of convex relaxations. Among these, the linx relaxation proposed by \cite{anstreicher_efficient_2020} provides one of the state-of-the-art relaxation bounds. Building on this, various enhancement techniques have been developed to further strengthen the relaxation bound quality. A recent advancement is the double-scaling approach, which, when applied to the linx relaxation, results in the following convex-concave saddle point formulation \citep{shen_majorization_2026}:
\begin{align*}
  \begin{split}
  & \min_{\x\in\cX}\max_{\bm\rho,\bm\omega\in\R^d}\left\{\phi(\x,\bm\rho,\bm\omega) \coloneq \frac12\langle\x,\bm\rho\rangle + \frac12\langle\ones-\x,\bm\omega\rangle \right. \\
  & \left.-\frac12\log\det(\bC\Diag(\exp(\bm\rho))\Diag(\x)\bC+\Diag(\exp(\bm\omega)){\Diag(\ones-\x)} ) \right\}, 
  \end{split}
\end{align*}
{where $\cX=\set{\x\in[0,1]^d:\, \ones^\top\x=s}$. Here, $\x$ models the principal submatrix selection, and the variables $\bm\rho,\bm\omega$ model the scaling parameters that are aimed at further strengthening the linx relaxation.}

\begin{remark}\label{rem:local-lip}
  The partial gradient $\nabla_\x\phi$ is neither bounded nor Lipschitz continuous over its domain. This can be illustrated by considering a diagonal matrix $\bC=\Diag(c_1,\dots,c_d)$: 
  \begin{align*}
    \frac{\partial\phi}{\partial x_i}(\x,\bm\rho,\bm\omega)
    &= -\frac12\cdot\frac{c_i^2\exp(\rho_i) - \exp(\omega_i)}{c_i^2\exp(\rho_i)x_i+\exp(\omega_i)(1-x_i)} + \frac12\rho_i - \frac12\omega_i \\
    &= -\frac12\cdot\frac{c_i^2\exp(\rho_i-\omega_i)-1}{c_i^2\exp(\rho_i-\omega_i)x_i+(1-x_i)} + \frac12(\rho_i-\omega_i). 
  \end{align*}
  When $x_i=1$, as $\rho_i-\omega_i\to-\infty$, {$ \frac{\partial\phi}{\partial x_i}(\x,\bm\rho,\bm\omega)$} grows exponentially fast, and therefore is neither bounded nor Lipschitz continuous w.r.t.\ $\bm\omega,\bm\rho\in\R^d$. On the other hand, it is locally Lipschitz continuous over its domain, as it is continuously differentiable. 
  \qed
\end{remark}

{Based on \cref{rem:local-lip}, the operator associated with this problem is neither bounded nor Lipschitz continuous. As a result, other than our algorithms, only \texttt{aGRAAL} comes with a provable-yet-ergodic convergence guarantee for this setting;  and all others do not admit any convergence guarantees and their stepsize selection is completely heuristic here.} Notably, even for the simpler g-scaled linx variant, where a convex-concave structure was already known to be present, prior work \citep{chen_generalized_2024} relied on a general-purpose nonconvex solver.

\paragraph{Problem data.} 
We consider the benchmark covariance matrix $\bC\in\R^{d\times d}$ with $d=124$ drawn from standard datasets used in environmental monitoring network redesign studies \citep{hoffman_new_2001}. This matrix has been widely used in the literature as a standard test instance for the MESP \citep{ko_exact_1995,lee_constrained_1998,anstreicher_using_1999,hoffman_new_2001,lee_linear_2003,anstreicher_masked_2004,burer_solving_2007,anstreicher_maximum-entropy_2018,li_best_2023,chen_generalized_2024}. For this covariance matrix, we generate a set of test instances by varying the subset size parameter $s=30,40,\dots,100$. 

\paragraph{Implementation details.}
All algorithms are initialized by setting $\x=\frac{s}{d}\ones$, i.e., the center of its domain $\cX$, and all scaling parameters are initialized at $0$, i.e., $\bm\rho=\bm\omega=\bf0$. The initial stepsizes are set to $\eta_0=0.1$, except for standard  \texttt{EG} with $\eta=0.05$. 
For this problem, the Euclidean projection onto the primal domain $\cX$ can be computed efficiently according to \cref{lem:project}. 
\begin{lemma}[{\cite{salem_no-regret_2023}}]\label{lem:project}
  Let $\cX=\set{\x\in[0,1]^d:\ones^\top\x=s}$ be the capped simplex. 
    In the Euclidean setting, the projection onto $\cX$ can be computed within $O(d^2)$ operations, and the domain diameter is $\Omega:=\frac12(s-s^2/d)$. 
\end{lemma}

{To compare the algorithm performance, we compute and monitor the natural residual $R_{0.01}(\z_t)$ for those with last-iterate convergence guarantees, or $R_{0.01}(\bar\z_t)$ for those with ergodic convergence, where $\bar\z_t$ denotes the (weighted) average of iterates. }

\paragraph{Performance comparison.}
\label{sec:experiment-nonsmooth}
{We report the corresponding convergence performances as well as the statistics on the number of iterations to reach to $\varepsilon$ accuracy and the resulting solution time} in \cref{fig:linx-tol}. Compared to other problem classes such as LASSO or group fairness classification, while the performance difference between ergodic and last-iterate algorithms is less pronounced for certain subset sizes $s$, last-iterate algorithms generally exhibit faster convergence. Moreover, ergodic algorithms show higher variance across different values of $s$, in contrast to the stable performance of the last-iterate algorithms. The convergence plots show a clear gap between our algorithms and the closest competitors,  \texttt{Adapt EG} and standard  \texttt{EG} {(neither of these have any theoretical convergence guarantees in this setting)}, which widens as the number of iterations increases. Although we do not have formal guarantees for \cref{alg:extragradient-adaptive} under local Lipschitz continuity, its numerical performance suggests a promising direction for future research. 
Both \cref{alg:extragradient-backtracking,alg:extragradient-backtracking-old}, which have a solid theoretical foundation for this problem, match the iteration count of the line-search-free \cref{alg:extragradient-adaptive} to reach target precision $\varepsilon$. While \cref{alg:extragradient-backtracking-old} achieves this at the cost of roughly doubling the solution time due to the stepsize increase trick discussed in \cref{rem:stepsize-increase-trick}, \cref{alg:extragradient-backtracking} has minimal additional overhead and achieves a solution time comparable to that of \cref{alg:extragradient-adaptive}. Overall, \cref{alg:extragradient-adaptive,alg:extragradient-backtracking} achieve the fastest convergence across all instances and subset sizes, and even though it it is slightly slower, \cref{alg:extragradient-backtracking-old} still outperforms all algorithms from the literature in all instances.

\begin{figure}[htbp]
  \centering
  \includegraphics[width=.49\textwidth]{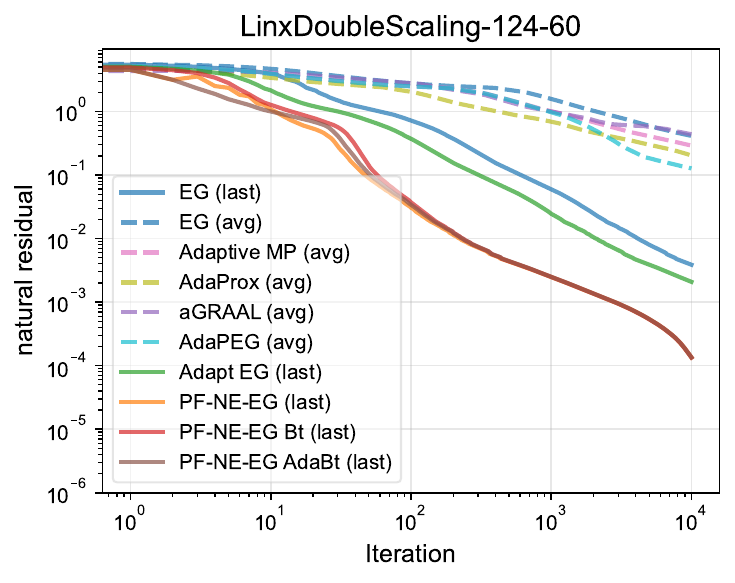}
  \includegraphics[width=.49\textwidth]{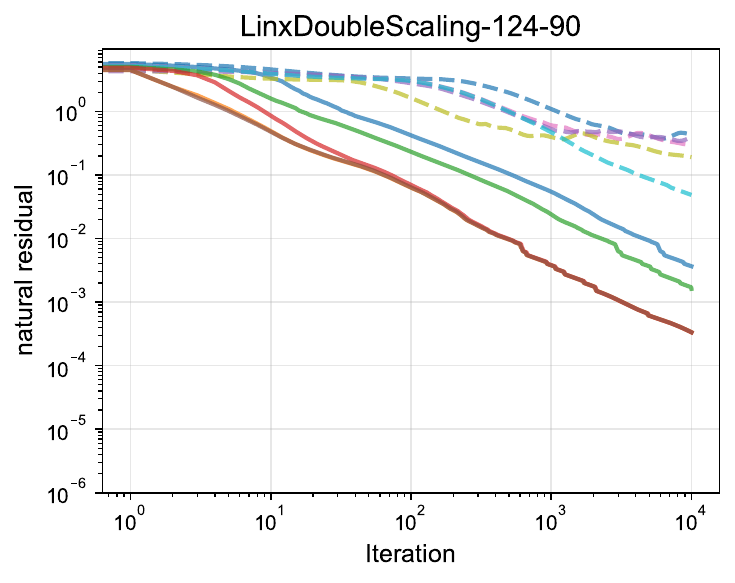}
  \includegraphics[width=.49\textwidth]{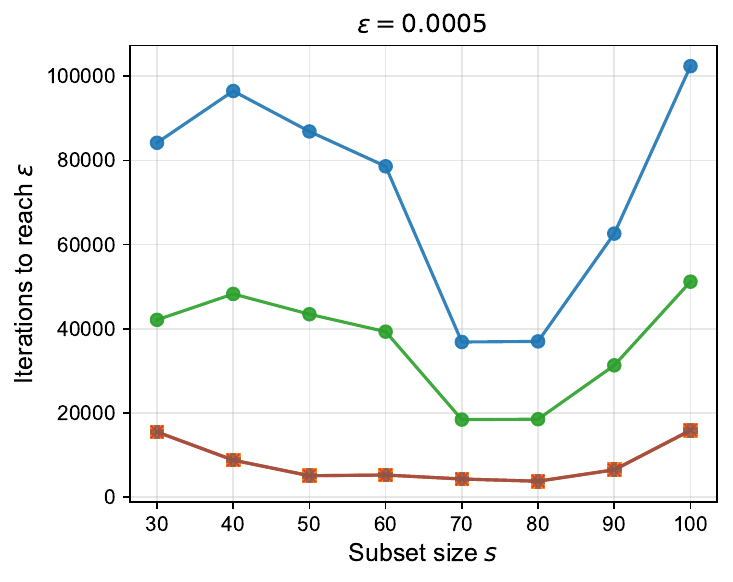}
  \includegraphics[width=.49\textwidth]{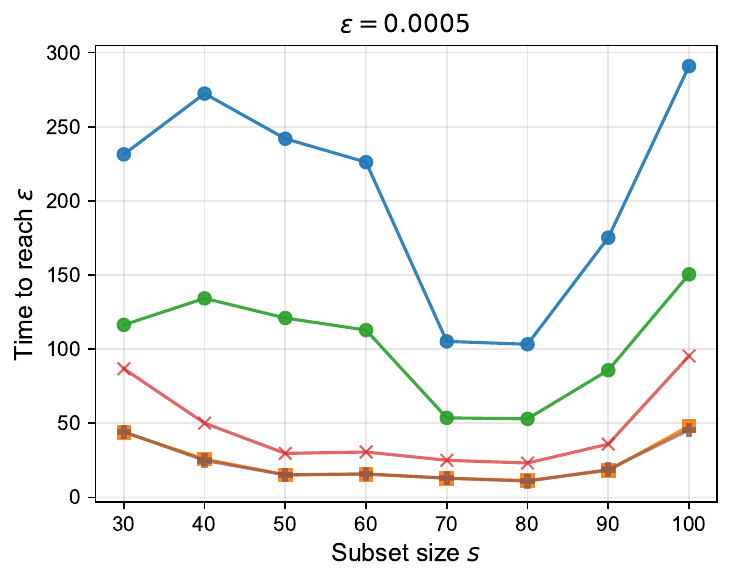}
  \caption{Comparison for convergence and iteration count and solution time (seconds) to reach the desired tolerance for solving linx double-scaling instances. }
  \label{fig:linx-tol}
\end{figure}

\section*{Acknowledgements}
This research was supported in part by AFOSR [Grant FA9550-22-1-0365].

\bibliographystyle{plainnat}
\bibliography{ref}

\appendix

\section{Standard backtracking line search}
\label{sec:appendix}
\label{sec:alg-backtracking-old}

In this section, we introduce an extragradient-type algorithm with standard backtracking line search, formally described in \cref{alg:extragradient-backtracking-old}. Like \cref{alg:extragradient-backtracking}, it is designed for operators with local Lipschitz continuity. While its theoretical convergence rate matches that of \cref{alg:extragradient-backtracking}, the standard backtracking line search procedure can lead to monotonically decreasing stepsizes that are overly conservative. We analyze \cref{alg:extragradient-backtracking-old} as a robust baseline for handling local Lipschitz continuity.

\begin{algorithm}[htbp]
  \caption{Parameter-free non-ergodic extragradient (PF-NE-EG) algorithm with standard backtracking line search}
  \label{alg:extragradient-backtracking-old}
  \begin{algorithmic}
    \Require Initial solution $\z_0\in\cZ$, initial stepsize $\eta_0>0$, $\theta\in(0,1)$, and $\rho\in(0,1)$. 
    \For{$t=1,2,\dots,T$}
      \State Step 1. Stepsize selection: Starting with $\eta=\eta_{t-1}$, decrease it by a factor of $\rho$ iteratively until it satisfies the conditions
      \begin{align}\label{eq:backtracking-old}
        \eta\frac{\|F(\w_t(\eta))-F(\z_t)\|_2}{\|\w_t(\eta)-\z_t\|_2} &\leq \theta < 1, \\
        \eta\frac{\|F(\w_t(\eta))-F(\z_{t+1}(\eta))\|_2}{\|\w_t(\eta)-\z_{t+1}(\eta)\|_2} &\leq 1,\quad\text{if }\w_t(\eta)\neq\z_{t+1}(\eta), \label{eq:backtracking-old-2}
      \end{align}
      where $\w_t(\eta)\coloneqq\Proj_{\cZ}(\z_t-\eta F(\z_t))$ and $\z_{t+1}(\eta)\coloneqq\Proj_{\cZ}(\z_t-\eta F(\w_t(\eta)))$, unless $\w_{t}(\eta)=\z_{t}$, in which case we stop and return $\z_t$. Set  $\eta_t=\eta$. 
      \State Step 2. Extragradient update:
        \begin{align*}
          \w_t &= \Proj_{\cZ}(\z_t-\eta_tF(\z_t)), \\
          \z_{t+1} &= \Proj_{\cZ}(\z_t-\eta_tF(\w_t)).
        \end{align*}
    \EndFor
    \Ensure $\z_T$.
  \end{algorithmic}
\end{algorithm}

The well-definedness and convergence rate analysis for \cref{alg:extragradient-backtracking-old} follow exactly the same proofs as those for \cref{alg:extragradient-backtracking}, presented in \cref{lem:finite-backtracking,prop:extragradient-backtracking}. Specifically, under \cref{assum:local-smooth}, the backtracking line search procedure in \cref{alg:extragradient-backtracking-old} stops in finite time with $\eta_t>0$ at each iteration, and the algorithm achieves an $o(1/\sqrt{T})$ convergence in the extragradient residual. The primary difference in the analysis lies in the total number of backtracking steps, as shown in \cref{lem:finite-backtracking-old}.

\begin{lemma}\label{lem:finite-backtracking-old}
  Suppose \cref{assum:local-smooth} holds. The backtracking line search procedure in \cref{alg:extragradient-backtracking-old} stops decreasing the stepsize within finitely many operations throughout all the iterations. In particular, there exists $\bar\eta>0$ such that $\eta_t\geq\bar\eta$ for all $t\in\N$. 
\end{lemma}

\begin{proof}
  From \cref{lem:descent}, we have 
  \begin{align*}
    \|\z_{t+1} - \z_*\|_2^2 
    &\leq \|\z_t - \z_*\|_2^2 - (1 - \eta_tL_t)[\|\z_t-\w_t\|_2^2 + \|\z_{t+1} - \w_t\|_2^2]. 
  \end{align*}
  By our line search procedure, we have $\eta_tL_t\leq\theta<1$, thus $\|\z_t-\z_*\|_2 \leq \|\z_0-\z_*\|_2$. This means $\{\z_t\}_{t\in\N}\subseteq\cB(\z_0,\z_*)\coloneqq\{\z\in\cZ:\|\z-\z_*\|_2\leq\|\z_0-\z_*\|_2\}$ is bounded. The local Lipschitz continuity of $F$ implies its Lipschitz continuity over any compact set (see \cite[Theorem 2.1.6]{cobzas_lipschitz_2019}), thus there exists Lipschitz constant $L(\z_0,\z_*)>0$ such that 
  \begin{align}\label{eq:local-lip-old}
    \|F(\z')-F(\z)\|_2 \leq L(\z_0,\z_*) \|\z'-\z\|_2, 
  \end{align}
  for any $\z,\z'\in\cZ$ such that $\|\z-\z_*\|_2,\|\z'-\z_*\|_2\leq\|\z_0-\z_*\|_2+1$. This further implies that there exists a constant $G(\z_0,\z_*)>0$ such that $\|F(\z)\|_2\leq G(\z_0,\z_*)$ for any $\z\in\cZ$ such that $\|\z-\z_*\|_2 \leq \|\z_0-\z_*\|_2+1$. 
  
  Case 1. If $\eta_{t-1}$ never satisfies $\eta_{t-1}L(\z_0,\z_*)\leq\theta$ and $\eta_{t-1}G(\z_0,\z_*)\leq1$, then $\eta_t$ is decreased only finitely many times throughout all the iterations of the algorithm and we stopped due to $\w_t(\eta_t)=\z_t$. 
  
  Case 2. If $\eta_{t-1}$ is sufficiently decreased such that $\eta_{t-1}L(\z_0,\z_*)\leq\theta$ and $\eta_{t-1}G(\z_0,\z_*)\leq1$, by definition of $\w_t(\eta)=\Proj_{\cZ}(\z_t-\eta F(\z_t))$, $\z_*\in\cZ$, and nonexpansiveness of projection, we have 
  \begin{equation}
  \begin{aligned}
    \|\w_{t}(\eta_{t-1})-\z_*\|_2 
    &\leq \|\z_{t}-\eta_{t-1}F(\z_{t})-\z_*\|_2 \\
    &\leq \|\z_{t}-\z_*\|_2+\eta_{t-1}\|F(\z_{t})\|_2 \\
    &\leq \|\z_{t}-\z_*\|_2+\eta_{t-1}G(\z_0,\z_*)
    \leq \|\z_0-\z_*\|_2+1. \label{eq:w-in-neighborhood}
  \end{aligned} 
  \end{equation}
  Therefore, \eqref{eq:local-lip-old} holds for $\z'=\w_t(\eta_{t-1})$ and $\z=\z_t$, and $\eta_{t-1}\frac{\|F(\w_{t}(\eta_{t-1}))-F(\z_{t})\|_2}{\|\w_{t}(\eta_{t-1})-\z_{t}\|_2} \leq \eta_{t-1}L(\z_0,\z_*) \leq\theta$. Moreover, using the definition of $\z_{t+1}(\eta)=\Proj_{\cZ}(\z_t-\eta F(\w_t(\eta)))$, $\z_*\in\cZ$, and nonexpansiveness of projection, we have 
  \begin{align*}
    \|\z_{t+1}(\eta_{t-1})-\z_*\|_2
    &\leq \|\z_{t}-\eta_{t-1}F(\w_{t}(\eta_{t-1}))-\z_*\|_2  \\
    &\leq \|\z_{t}-\z_*\|_2+\eta_{t-1}\|F(\w_{t}(\eta_{t-1}))\|_2 \\
    &\leq \|\z_{t}-\z_*\|_2+\eta_{t-1}G(\z_0,\z_*)
    \leq \|\z_0-\z_*\|_2+1, 
  \end{align*}
  where the third inequality follows from the definition of $G(\z_0,\z_*)$ and \eqref{eq:w-in-neighborhood}. 
  Thus, \eqref{eq:local-lip-old} holds for $\z'=\w_t(\eta_{t-1})$ and $\z=\z_{t+1}(\eta_{t-1})$, and $\eta_{t-1}\frac{\|F(\w_{t}(\eta_{t-1}))-F(\z_{t+1}(\eta_{t-1}))\|_2}{\|\w_{t}(\eta_{t-1})-\z_{t+1}(\eta_{t-1})\|_2}\leq\eta_{t-1}L(\z_0,\z_*)\leq1$ if $\w_t(\eta_{t-1})\neq\z_{t+1}(\eta_{t-1})$. This means \eqref{eq:backtracking-old} and \eqref{eq:backtracking-old-2} are satisfied by $\eta=\eta_{t-1}$, therefore $\eta_t=\eta_{t-1}$ is no longer decreased in the subsequent iterations.
\end{proof}

\begin{remark}\label{rem:stepsize-increase-trick}
The limitation of non-increasing stepsize of the standard backtracking procedure can be resolved by a simple trick (see e.g., \citep{lu_primal-dual_2025}). The idea is that, at the start of a new iteration $t$, instead of reusing the stepsize $\eta=\eta_{t-1}$ from the previous iteration, we first proactively increase it by a factor $\rho^{-1} > 1$. This allows the stepsize to increase as needed, at the cost of at most two more operator evaluations per iteration. However, it no longer guarantees a finite total number of backtracking steps. As the iterates approach the VI solution, the local Lipschitz constant stabilizes, whereas the aggressive initial increase at each iteration pushes the stepsize beyond its ideal choice (i.e., the reciprocal of the local Lipschitz constant), which costs an additional backtracking step. This explains why, in the numerical experiments in \cref{sec:numerical}, \cref{alg:extragradient-backtracking-old} requires roughly twice the solution time of \cref{alg:extragradient-adaptive,alg:extragradient-backtracking}.
To keep our discussion simple, 
we analyze \cref{alg:extragradient-backtracking-old} without this modification.
\end{remark} 

\end{document}